\documentclass[reqno]{amsart}
\usepackage[utf8]{inputenc}
\usepackage{mathtools}
\usepackage{amssymb}
\usepackage{enumitem}
\usepackage{xcolor}
\usepackage[hidelinks]{hyperref}
\hypersetup{
colorlinks=true, 
linkcolor= blue, 
citecolor=blue
}
\usepackage{graphicx} 
\usepackage{amsaddr}

\DeclarePairedDelimiter{\abs}{\lvert}{\rvert}
\DeclarePairedDelimiter{\norm}{\lVert}{\rVert}
\DeclarePairedDelimiter{\floor}{\lfloor}{\rfloor}
\DeclarePairedDelimiter{\inner}{\langle}{\rangle}
\DeclarePairedDelimiter{\set}{\lbrace}{\rbrace}
\DeclarePairedDelimiter{\br}{(}{)}
\DeclarePairedDelimiter{\sbr}{[}{]}
\DeclarePairedDelimiter{\seqn}{(}{)}
\DeclarePairedDelimiter{\seqk}{(}{)}

\DeclarePairedDelimiter{\ang}{\langle}{\rangle}

\DeclareMathOperator*{\esssup}{ess\,sup}
\DeclareMathOperator*{\essinf}{ess\,inf}

\newcommand{\sigess}{\sigma_{e}}
\newcommand{\sigessh}{\hat{\sigma}_{e}}
\newcommand{\sigdis}{\sigma_{d}}
\newcommand{\sigpoll}{\sigma_{\rm poll}}
\newcommand{\wto}{\rightharpoonup}
\newcommand{\gsrto}{\xrightarrow{\rm{gsr}}}
\newcommand{\srto}{\xrightarrow{\rm{sr}}}
\newcommand{\sto}{\xrightarrow{\rm{s}}}
\newcommand{\bs}{\backslash}

\newcommand{\R}{\mathbb R}
\newcommand{\C}{\mathbb C}
\newcommand{\N}{\mathbb N}

\renewcommand{\d}{\, \text{d}}

\newcommand{\T}{T}

\newcommand{\Sp}{S_\mathfrak{p}(\seqn{R_n})}
\newcommand{\Sr}{S_\mathfrak{r}}

\newcommand{\psip}{\psi_+}
\newcommand{\psim}{\psi_-}
\newcommand{\tpsip}{\tilde{\psi}_+}
\newcommand{\tpsim}{\tilde{\psi}_-}
\newcommand{\tpsipd}{\tilde{\psi}^d_+}
\newcommand{\tpsimd}{\tilde{\psi}^d_-}

\renewcommand{\geq}{\geqslant}
\renewcommand{\leq}{\leqslant}
\renewcommand{\epsilon}{\varepsilon}
\renewcommand{\subseteq}{\subset} 

\newtheorem{theorem}{Theorem}[section]
\newtheorem{proposition}[theorem]{Proposition}
\newtheorem{lemma}[theorem]{Lemma}
\newtheorem{corollary}[theorem]{Corollary}
\theoremstyle{remark}
\newtheorem{remark}[theorem]{Remark}
\theoremstyle{definition}
\newtheorem{assumption}{Assumption}
\theoremstyle{definition}

\newtheorem{definition}[theorem]{Definition}
\theoremstyle{definition}
\newtheorem{example}[theorem]{Example}

\numberwithin{equation}{section}

\title[Spectral inclusion and pollution for a class of perturbations]{Spectral inclusion and pollution for a class of dissipative perturbations}

\author{Alexei Stepanenko}
\address{Cardiff University, School of Mathematics, Senghennydd Road, Cardiff, UK,CF24 4AG}

\date{\today}

\email{stepanenkoa@cardiff.ac.uk}



\begin{document}

\begin{abstract}

Spectral inclusion and spectral pollution results are proved for sequences of linear operators of the form $T_0 + i \gamma s_n$ on a Hilbert space,  where $s_n$ is strongly convergent to the identity operator and $\gamma > 0$. 
We work in both an abstract setting and a more concrete Sturm-Liouville framework.
The results provide rigorous justification for a method of computing eigenvalues in spectral gaps.  

\end{abstract}

\maketitle

\section{Introduction}

In this paper, we study the eigenvalues of linear operators under a certain class of perturbations
with an emphasis on Schr\"odinger operators of the form,
\begin{equation}\label{eq:schrod-half-db}
    T_R = - \frac{\d^2}{\d x^2} + q + i \gamma \chi_{[0,R]} \qquad \text{on }L^2(0,\infty),
\end{equation}
where  $q$ is a possibly complex-valued function and $\chi$ is the characteristic function.
Specifically, we are concerned with how the eigenvalues of $T_R$ approximate the spectrum of the \textit{limit operator} $T = - \frac{\d^2}{\d x^2} + q + i \gamma$. 
As well as giving a precise account for the case of Schr\"odinger operators $T_R$ with the \textit{background potential} $q$ either in $L^1$ or eventually real periodic, we give general results for abstract operators of this form, utilising the notion of limiting essential spectrum recently introduced by B\"ogli (2018) \cite{SabineLocal2016}.

It is well known that the numerical approximation of the spectra of linear operators is often complicated by the possible presence of \textit{spectral pollution} \cite{poll_boffi2000,poll_method_davies2004,poll_kutz1984,poll_rappaz1997}.
The primary motivation for this paper is the justification of the \textit{dissipative barrier method}, designed to circumvent such issues.

The perturbations we consider belong to a class of operators which are often referred to as \textit{complex absorbing potentials} in the context of Schr\"{o}dinger operators. 
These arise in the study of the damped wave equation \cite{damped_christianson2009, damped_christanson2014, damped_frietas2018}, in the computation of resonances in quantum chemistry \cite{cap_riss1993, cap_stefanov2005,cap_zworski2018} and in the study of resonances in quantum chaos \cite{nonnenmacher2009,nonnenmacher2015}.

\subsection{Spectral Inclusion and Pollution}

Suppose that we are interested in approximating the spectrum of a (linear) operator $H$ on a Hilbert space $\mathcal{H}$ with domain $D(H)$.
Let $\seqn{H_n}$ be a sequence of operators on $\mathcal{H}$ whose spectra $\sigma(H_n)$ we hope will approximate the spectrum $\sigma(H)$ of $H$ as $n \to \infty$. 
The \textit{limiting spectrum} of $\seqn{H_n}$ is defined by 
\begin{equation}\label{eq:lim-spec-defn}
    \sigma(\seqn{H_n}) = \set{\lambda \in \C: \exists I \subseteq \N \text{ infinite},\, \exists \lambda_n \in \sigma(H_n),n\in I\text{ with } \lambda_n \to \lambda }.
\end{equation}
$\seqn{H_n}$ is said to be \textit{spectrally inclusive} for $H$ in some $\Omega \subseteq \C$ if
\begin{equation}\label{eq:spec-incl-defn}
    \sigma(H)\cap\Omega \subseteq \sigma(\seqn{H_n}).
\end{equation}
The set of \textit{spectral pollution} for $\seqn{H_n}$ with respect to $H$ is defined by 
\begin{equation}\label{eq:spec-poll-defn}
    \sigpoll(\seqn{H_n}) = \set{ \lambda \in \sigma(\seqn{H_n}): \lambda \notin \sigma(H)}.
\end{equation}
In order to reliably approximate the spectrum of $H$ in $\Omega \subseteq \C$ using $\seqn{H_n}$, we require that there is no spectral pollution in $\Omega$, $\sigpoll(\seqn{H_n})\cap\Omega = \emptyset$, and that $\seqn{H_n}$ is spectrally inclusive for $H$ in $\Omega$. If this holds, we say that $\seqn{H_n}$ is \textit{spectrally exact} for $H$ in $\Omega$. 

A typical scenario in which the set of spectral pollution may be non-empty is one in which the essential spectrum $\sigess(H)$ of $H$ has a band-gap structure and the operators $H_n$ have compact resolvents (i.e. $H_n$ have purely discrete spectra). 
For this reason, spectral pollution often causes issues for the numerical computation of eigenvalues in spectral gaps.
Various methods have been proposed to deal with such issues, we mention for instance \cite{method_boulton2007,poll_method_davies2004,method_hansen2008,method_levitin2004,method_lewin2010,method_zimmermann1995}.
We focus on one such method, which involves perturbing the operator of interest such as to move the spectrum, in a predictable way, away from the set of spectral pollution caused by numerical discretisation \cite{marletta2012eigenvalues}. 

\subsection{Dissipative Barrier Method}

Let us now describe this method. 
Let $T_0$ be a self-adjoint operator on a Hilbert space $\mathcal{H}$; suppose we are interested in numerically computing the spectrum of $T_0$. 
A dissipative barrier method for $T_0$ is defined by a bounded sequence of self-adjoint, $T_0$-compact operators $\seqn{s_n}$ tending strongly to the identity operator on $\mathcal{H}$.
If $\mathcal{H} = L^2(0,\infty)$, for instance, a typical choice for $s_n$ would be $\chi_{[0,n]}$.
Define the \textit{perturbed operators} by
\begin{equation}\label{eq:Tn-intro}
    T_n = T_0 + i \gamma s_n\qquad(n \in \N)
\end{equation}
where $\gamma > 0$.
The \textit{limit operator} $T$ is defined by $T = T_0 + i \gamma$.
The spectrum of $T_0$ is exactly encoded in the spectrum of $T$ since $\sigma(T) = \sigma(T_0) + i\gamma$.

Under appropriate additional conditions on $T_0$ and $s_n$, it can be proved that there exist spectrally inclusive numerical methods for the computation of $\sigma(T_n)$ for fixed $n$ \cite{db_aljawi2020,db_Bogli2017,db_marletta2010ntd,db_naboko2014,marletta2012eigenvalues,db_strauss2014}. Furthermore, any spectral pollution for these numerical methods lies on $\R$, away from $\sigma(T)$ uniformly in $n$.
The recently introduced notion of essential numerical range for unbounded operators can be used to prove general results of this form (see Theorems 4.5, 6.1 and 7.1 in \cite{EssNumRan2020}).
Thanks to such numerical methods for $\sigma(T_n)$, if $\seqn{T_n}$ can be shown to be spectrally exact for $T$ in an open neighbourhood in $\C$ of a closed subset $ i \gamma + I\subseteq i \gamma + \R$, then in principle one can reliably numerically compute the spectrum of $T_0$ in $I$. 

\subsection{Analysis of Expanding Barriers}

The aim of this paper is to provide spectral inclusion and spectral pollution results for sequences of operators of the form \eqref{eq:Tn-intro}.

In Section \ref{sec:spec-poll}, we work in an abstract setting, utilising the limiting essential spectrum $\sigess(\seqn{T_n})$ \cite{SabineLocal2016}, which is a set enclosing the regions in $\C$ where spectral exactness for $\seqn{T_n}$ with respect to $T$ may fail.
With additional assumptions on the operators $s_n$, for instance that they are projection operators, we prove new types of non-convex enclosures for $\sigess(\seqn{T_n})$ and conclude for these cases that $\seqn{T_n}$ is spectrally exact for $T$ in an open neighbourhood of any eigenvalue of $T$. 
The paper \cite{db_hinchcliffe2016} gives a similar spectral exactness conclusion for the case that $\seqn{s_n}$ are projection operators. 
However, as well as including different classes of perturbations $\seqn{s_n}$, both the statement and the proof of our results in Section \ref{sec:spec-poll} are far simpler than those of \cite{db_hinchcliffe2016}, owing to the use of the limiting essential spectrum.

The remainder of the paper is devoted to a more precise analysis for the case of Sturm-Liouville operators on the half-line. 
Our results in Sections \ref{sec:spec-incl} and \ref{sec:ess-incl} apply to operators for which the solutions of the corresponding Sturm-Liouville equation satisfy a certain decomposition.
In particular, this decomposition is easily shown to be satisfied by Schr\"odinger operators $T_R$ of the form \eqref{eq:schrod-half-db} with the background potential $q$ either in $L^1$ or real eventually periodic. 
In Section \ref{sec:spec-incl}, we show that any eigenvalue of the limit operator $T \equiv T_0 + i \gamma \equiv - \d^2 / \d x^2 + q + i \gamma$ for these cases is approximated by the spectrum of $T_R$  with exponentially small error as $R \to \infty$. A similar result was proved in \cite[Theorem 10]{marletta2012eigenvalues}, but only for $\gamma$ sufficiently small. In Section \ref{sec:ess-incl} we show that the essential spectrum of $T$ is approximated by the eigenvalues of $T_R$ with an error of order $O(1/R)$\footnote{Although band-ends and embedded resonances may have a different rate of convergence.}. 
The latter result is the first of its type to be reported. 

We also characterise the set of spectral pollution for the two cases of perturbed Schr\"odinger operators $T_R$.
Let $\seqn{R_n}\subset \R_+$ be any sequence such that $R_n \to \infty$.
Since the dissipative barrier perturbations $i \gamma \chi_{[0,R_n]}$ are relatively compact, the essential spectrum $\sigess(T_0)$ is contained in the spectral pollution $\sigpoll(\seqn{T_{R_n}})$ by Weyl's Theorem\footnote{With the possible exception of a few isolated points if $T_0$ is non-self-adjoint.}.
Note that this is in contrast to typical examples of spectral pollution, due to numerical discretisation, which are caused by spurious eigenvalues.
It is shown in Section \ref{sec:spec-incl} that $\sigess(T_0)$ is the only possible source of spectral pollution for the case $q \in L^1$. 
We encourage the reader to inspect Figures \ref{fig:1} and \ref{fig:2} in Section \ref{sec:num-ex}, which illustrate the eigenvalues of $T_R$ for this case.
For  $q$ eventually real periodic, the set of spectral pollution outside $\sigess(T_0)$ is enclosed in the set of zeros of a certain analytic function constructed from solutions of (time-independent) Schr\"odinger equations.
In fact, we prove that these zeros are contained inside the limiting essential spectrum $\sigess(\seqn{T_{R_n}})$.
Figure \ref{fig:3} in Section \ref{sec:num-ex} shows how spectral pollution may occur in this second case.

\subsection{Summary of Results}
 
The definitions of the essential spectrum $\sigess(H)$ and the discrete spectrum $\sigdis(H)$ for an operator $H$ are given by equations \eqref{eq:sig-ess-defn} and \eqref{eq:sigdis-conv} below.

\subsubsection*{Limiting Essential Spectrum and Spectral Pollution}

In Section \ref{sec:spec-poll}, we consider a self-adjoint operator $T_0$ on Hilbert space $\mathcal{H}$.
It is assumed that the operators $s_n$ ($n \in \N$) on $\mathcal{H}$ are self-adjoint, tend strongly to the identity operator as $n \to \infty$ and are bounded independently of $n$.
For $\gamma > 0$, we define the perturbed operators $T_n$ $(n \in \N)$ by \eqref{eq:Tn-intro} and the limit operator by $T = T_0 + i \gamma$.

The main tool in this section is the notion of \textit{limiting essential spectrum} $\sigess(\seqn{T_n})$ (see Definition \ref{def:lim-ess}). 
The results of \cite{SabineLocal2016} show that (Corollary \ref{col:exactness})
\begin{equation*}
    \seqn{T_n}\text{ is spectrally exact for }T\text{ in }\C \bs \sbr*{ \sigess(\seqn{T_n})\cup\sigess(\seqn{T^*_n})^* \cup \sigess(T)}.
\end{equation*}
The \textit{limiting essential numerical range} $W_e(\seqn{T_n})$ of $\seqn{T_n}$ (see Definition \ref{def:lim-ess-num}), introduced by B\"ogli, Marletta and Tretter (2020), is a convex set which in our set-up satisfies (Propositions \ref{prop:ess-num} and \ref{prop:lim-num-est})
\begin{equation*}
    \sigess(\seqn{T_n}) \cup \sigess(\seqn{T^*_n})^* \subset W_e(\seqn{T_n}) \subset \sbr*{ \text{\rm conv}(\sigessh(T_0)) \bs \set{\pm \infty} } \times i \gamma[s_-,s_+],
\end{equation*}
where $\sigessh(T_0)$ denotes the extended essential spectrum of $T_0$ (see Definition \ref{def:sig-ess-extended}) and  $s_\pm \in \R$ (defined by \eqref{eq:sp-and-sm}) satisfy $s_- - \epsilon \leq s_n \leq s_+ + \epsilon$ for any $\epsilon > 0$ and large enough $n$. 

The main results of Section \ref{sec:spec-poll} are non-convex enclosures for $\sigess(\seqn{T_n})$ complementing the enclosure provided by $W_e(\seqn{T_n})$.
\begin{enumerate}[label=(A)]
    \item (Theorem \ref{thm:encl}) If $s_n$ is a projection operator for all $n$, that is $s_n^2 = s_n$, then $\sigess(\seqn{T_n}) \cup \sigess(\seqn{T^*_n})^* \subset \Gamma_a = \Gamma_a(\sigess(T_0), \gamma)$, where
    \begin{equation}\label{eq:Gamma-a}
        \qquad \Gamma_a := \set*{\lambda \in \C : \Im(\lambda) \in [0,\gamma], \, \text{dist}(\Re(\lambda),\sigess(T_0)) \leq \sqrt{\Im(\lambda)(\gamma - \Im(\lambda))}}.
    \end{equation}
    If for any sequence $\seqn{u_n} \subset D(T_0)$ bounded in $\mathcal H$ with $\seqn{T_0 u_n}$ bounded in $\mathcal H$ we have
    \begin{equation*}
        \inner{s_n u_n, T_0 u_n} - \inner{T_0 u_n, s_n u_n} \to 0 \text{ \rm as }n \to \infty
    \end{equation*}
    (Assumption \ref{ass:loc-flat}) then $\sigess(\seqn{T_n}) \cup \sigess(\seqn{T^*_n})^* \subset \Gamma_b= \Gamma_b(\sigess(T_0), \gamma, s_\pm)$, where 
    \begin{equation*}
        \Gamma_b := \sigess(T_0) \times i \gamma [s_-,s_+].
    \end{equation*}
\end{enumerate}
In particular, if $s_n$ are projection operators or if Assumption \ref{ass:loc-flat} is satisfied then 
\begin{equation*}
    \seqn{T_n}\text{ is spectrally exact for }T\text{ in some open neighbourhood of any }\lambda \in \sigdis(T).
\end{equation*}
We clarify that by open neighbourhood we mean open neighbourhood in $\C$.
The enclosures $\Gamma_a$ and $\Gamma_b$ are illustrated in Figure \ref{fig:encl}.
Assumption \ref{ass:loc-flat} is verified for a class of perturbations for Schr\"odinger operators on Euclidean domains in Example \ref{ex:loc-flat}.

\subsubsection*{Second Order Operators on the Half-Line}

In Section \ref{sec:spec-incl}, we consider the case in which $T_0$ is a Sturm-Liouville operator on $L^2(0,\infty)$ and provide a more precise analysis compared to Section \ref{sec:spec-poll}. 
The Sturm-Liouville operator $T_0$ is allowed to have complex coefficients and is endowed with a complex mixed boundary condition at 0.  

We assume that for any $\lambda \in \C \bs \sigess(T_0)$, the solution space of the equation $\tilde{T}_0 u = \lambda u$ (here, $\tilde{T}_0$ is the differential expression corresponding to $T_0$) is spanned by solutions $\psi_\pm(\cdot,\lambda)$ admitting the decomposition
\begin{equation*}
    \psi_\pm(x,\lambda) = e^{\pm ik(\lambda)x} \tilde{\psi}_\pm(x,\lambda).
\end{equation*} 
Here, $k$ and $\tilde{\psi}_\pm(x,\cdot)$ are analytic functions on $\C \bs \sigess(T_0)$ with $\Im k >0$ and with $\tilde{\psi}_\pm(\cdot,\lambda)$ bounded.
A similar decomposition is required for $\psi'_\pm$ - see Assumption \ref{ass:exp} for the precise statement.

The perturbed operators in Section \ref{sec:spec-incl} are defined by
\begin{equation}\label{eq:TR-summary}
    T_R = T_0 + i \gamma \chi_{[0,R]}\qquad(R \in \R_+)
\end{equation}
where $\gamma \in \C \bs \set{0}$.
The limit operator is defined by $T = T_0 + i \gamma$.
Under these assumptions, for any $\seqn{R_n}$ with $R_n \to \infty$ , we construct a set $\Sp$ (equation \eqref{eq:Sp-defn}) and prove the following:
\begin{enumerate}[label=(B)]
    \item (Theorems \ref{thm:exp-conv} and \ref{thm:spec-poll}) For any eigenvalue $\lambda$ of $T$ with $\lambda \notin \Sp$ and $\lambda \notin \sigess(T_0)$, there exists eigenvalues $\lambda_n$ of $T_{R_n}$ ($n \in \N$) such that 
    \begin{equation*}
    |\lambda - \lambda_n| = O(e^{ - \beta R_n})\text{ as }n\to \infty     
    \end{equation*}
    for some $\beta > 0$ independent of $n$. 
    Furthermore, the set of spectral pollution for $\seqn{T_{R_n}}$ with respect to $T$ satisfies
    \begin{equation*}
        \sigpoll(\seqn{T_{R_n}}) \subseteq \sigess(T_0) \cup \Sp. 
    \end{equation*}
\end{enumerate}
The proofs utilise Rouch\'e's theorem applied to an analytic function (Lemma \ref{lem:fR}) whose zeros are the eigenvalues of $T_R$.
(B) implies that
\begin{equation*}
    \seqn{T_{R_n}}\text{ is spectrally exact for }T\text{ in }\C \bs \br*{\sigess(T_0) \cup \sigess(T) \cup \Sp}.
\end{equation*}
Assumption \ref{ass:exp} is verified in two cases:
\begin{itemize}
    \item (Examples \ref{ex:Lev-1} and \ref{ex:Lev-2}) $T_0$ is a Schr\"odinger operator with an $L^1$ potential. In this case, $\Sp = \emptyset$.
    \item (Examples \ref{ex:Floq-1} and \ref{ex:Floq-2}) $T_0$ is a Schr\"odinger operator with an eventually real $a$-periodic potential, $\gamma > 0$ and $R_n - R_{n-1} = a$ for all $n$. In this case, $\Sp$ is expressed as the zeros of a certain analytic function (equation \eqref{eq:Floq-Sp-zeros}). It is also proved that $\Sp \subset \sigess(\seqn{T_n})$.
    
\end{itemize}

\subsubsection*{Inclusion for the Essential Spectrum}

In Section \ref{sec:ess-incl}, we let $T_0$ be a Sturm-Liouville operator satisfying Assumption \ref{ass:exp}, as described above. 
In addition, we require that $\sigess(T_0)\subseteq \R $ and that $k$ and $\tilde{\psi}_+(x,\cdot)$, hence $\psi_+(x,\cdot)$, admit analytic continuations into an open neighbourhood of any point in the interior of $\sigess(T_0)$.
See Assumption \ref{ass:analytic} for the precise statement. 

The perturbed operators $T_R$ and the limit operator in Section \ref{sec:ess-incl} are defined by \eqref{eq:TR-summary} and $T = T_0 + i \gamma$ respectively, as in Section \ref{sec:spec-incl}.
We construct a set $\Sr \subseteq i \gamma + \R $ (equation \eqref{eq:Sr-defn}) and prove that: 
\begin{enumerate}[label=(C)]
    \item (Theorem \ref{thm:ess-incl}) For any $\mu$ in the interior of $\sigess(T_0)$ with $\mu + i \gamma \notin \Sr$, there exists eigenvalues $\lambda_R$ of $T_R$ ($R \in \R_+$) such that 
    \begin{equation*}
        |\lambda_R - (\mu + i \gamma)| = O\br*{\frac{1}{R}} \text{ as }R \to \infty.
    \end{equation*}
\end{enumerate}
The proof utilises Rouch\'e's theorem applied to an analytic function (Lemma \ref{lem:gR}) whose zeros are the eigenvalues of $T_R$.
In the case that
\begin{itemize}
    \item (Examples \ref{ex:Lev-3} and \ref{ex:Lev-4}) $T_0$ is a Schr\"odinger operator with an $L^1$ potential satisfying the Naimark condition or a dilation analyticity condition, or,
    \item (Examples \ref{ex:Floq-3} and \ref{ex:Floq-4}) $\gamma > 0$ and $T_0$ is a Schr\"odinger operator with a real, eventually periodic potential, endowed with a real mixed boundary condition at 0,
\end{itemize}
it is proven that Assumption \ref{ass:analytic} is satisfied and that 
\begin{equation*}
\mu + i \gamma \in \Sr \text{ if and only if }\mu \text{ is a resonance of }T_0\text{ embedded in }\sigess(T_0).
\end{equation*}
See equation \eqref{eq:Sr-0-defn} for the precise definition of a resonance used here.
For these cases, since resonances in the interior of $\sigess(T_0)$ are isolated, we can combine Theorem \ref{thm:ess-incl} with Theorem \ref{thm:spec-poll} and the characterisation of $\Sp$ to conclude that
\begin{equation*}
    \seqn{T_n}\text{ is spectrally exact for }T\text{ in some open neighbourhood of any }\mu \in \text{int}(\sigess(T)).
\end{equation*}

\subsection*{Notation and Conventions}
Let $\mathcal H$ be a separable Hilbert space with corresponding inner product $\inner{\cdot,\cdot}$ and norm $\norm{\cdot}$. Let $B_n \sto B$ as $n \to \infty$ denote strong convergence in $\mathcal H$ for bounded operators $B_n$ and $B$ on $\mathcal H$. Let $f_n \wto f$ as $n \to \infty$ denote weak convergence in $\mathcal{H}$ for $f_n,f\in \mathcal{H}$.
In this paper, we define the essential spectrum of an operator $H$ on $\mathcal{H}$ by
\begin{equation}\label{eq:sig-ess-defn}
    \sigess (H) = \set*{\lambda \in \mathbb C : \begin{matrix} \exists \seqn{u_n} \subset D(H)\textrm{ with }\norm{u_n} = 1,\\  u_n \wto 0,\, \norm{(H - \lambda) u_n} \to 0 \end{matrix}}
\end{equation}
which corresponds to $\sigma_{e2}$ in \cite{edmunds2018spectral}. The sequence $\seqn{u_n}$ appearing in \eqref{eq:sig-ess-defn} is referred to as a singular sequence. 
The discrete spectrum is defined by
\begin{equation}\label{eq:sigdis-conv}
    \sigdis(H) = \sigma(H) \bs \sigess(H).
\end{equation}

The convention we take with regards to the square-root function is to make the branch-cut along the positive semi-axis, so that $\Im \sqrt{z} \geq 0$ for all $z \in \C$.
We let $B_r(z)$ denote an open ball of radius $r > 0$ around a point $z \in \C$. 
In Sections \ref{sec:spec-incl} and \ref{sec:ess-incl}, $\psi'(x,z):=\frac{\d}{\d x}\psi(x,z)$.

\section{Limiting Essential Spectrum and Spectral Pollution}\label{sec:spec-poll}

In this section, we study spectral exactness for sequences of abstract operators $\seqn{T_n}$ of the form \eqref{eq:Tn-intro}. 
In Section \ref{subsec:review}, we briefly review the notions of limiting essential spectrum and essential numerical range. 
We refer the reader to \cite{SabineLocal2016} and \cite{EssNumRan2020} for a more detailed exposition. 
In Section \ref{subsec:discussion}, we discuss the application of limiting essential spectrum and essential numerical range to $\seqn{T_n}$.
In Section \ref{subsec:main}, we prove enclosures for the limiting essential spectrum of $\seqn{T_n}$. 

\subsection{Limiting Essential Spectrum and Numerical Range}\label{subsec:review}
 
Let $\mathcal{H}_n \subset \mathcal{H}$ ($n \in \N$) be closed subspaces and let $P_n: \mathcal{H} \to \mathcal{H}_n$ be the corresponding orthogonal projections. 
Assume that $P_n \sto I$.
Let $H$ and $H_n$ ($n \in \N$) be closed, densely-defined operators acting on $\mathcal{H}$ and $\mathcal{H}_n$ respectively. 

\begin{definition}\label{def:lim-ess}
 The \textit{limiting essential spectrum} of $\seqn{H_n}$ is defined by
\begin{equation}
    \sigess (\seqn{H_n}) = \set*{\lambda \in \mathbb C : \begin{matrix} \exists I\subseteq \mathbb N \textrm{ infinite},\,\exists u_n \in D(H_n), n \in I \textrm{ with }\\ \norm{u_n} = 1, \, u_n \wto 0,\, \norm{(H_n - \lambda) u_n} \to 0 \end{matrix}}. 
\end{equation}
\end{definition}

\begin{definition}\label{def:gsr}
 $\seqn{H_n}$ converges to $H$ in the \textit{generalised strong resolvent sense}, denoted by $H_n \gsrto H$, if 
 \begin{equation*}
    \exists n_0 \in \N: \exists \lambda_0 \in \bigcap_{n \geq n_0} \rho(H_n) \cap \rho(H): (H_n - \lambda_0)^{-1}P_n \sto (H - \lambda_0)^{-1}.
 \end{equation*}
\end{definition}

In the case that $\mathcal{H}_n = \mathcal{H}$ for all $n$, generalised strong resolvent convergence is equivalent to strong resolvent convergence and denoted by $H_n \srto H$. 

\begin{theorem}[{\cite[Theorem 2.3]{SabineLocal2016}}]\label{thm:bogli-local}
 If $H_n \gsrto H$ and $H^*_n \gsrto H^*$ then 
 \begin{equation}
    \sigpoll(\seqn{H_n}) \subset \sigess(\seqn{H_n}) \cup \sigess(\seqn{H^*_n})^*
 \end{equation}
and every isolated $ \lambda \in \sigma(H)$ outside $\sigess(\seqn{H_n}) \cup \sigess(\seqn{H^*_n})^*$ is approximated by $\seqn{H_n}$, that is, 
\begin{equation*}
    \set{\lambda \in \sigma(H): \lambda \text{ \rm isolated},\, \lambda \notin \sigess(\seqn{H_n}) \cup \sigess(\seqn{H^*_n})^*} \subset \sigma(\seqn{H_n}).
\end{equation*}
\end{theorem}

\begin{definition}\label{def:ess-num}
The \textit{essential numerical range} of $H$ is defined by
\begin{equation*}
    W_e(H) = \set{ \lambda \in \C : \exists \seqn{u_n} \subset D(H) \text{ with } \norm{u_n} = 1,\,u_n\wto0,\,\inner{H u_n, u_n} \to \lambda  }.
\end{equation*}
\end{definition}

\begin{definition}\label{def:lim-ess-num}
 The \textit{limiting essential numerical range} of $\seqn{H_n}$ is defined by
\begin{equation}
    W_e (\seqn{H_n}) = \set*{\lambda \in \mathbb C : \begin{matrix} \exists I\subseteq \mathbb N \textrm{ infinite},\,\exists u_n \in D(H_n), n \in I \textrm{ with }\\ \norm{u_n} = 1, \, u_n \wto 0,\, \inner{(H_n - \lambda) u_n,u_n} \to 0 \end{matrix}}. 
\end{equation}
\end{definition}

\begin{proposition}[{\cite[Proposition 5.6]{EssNumRan2020}}]\label{prop:ess-num}
    The limiting essential numerical range of $\seqn{H_n}$ is closed and convex with
    \begin{equation*}
     \text{\rm conv}(\sigess(\seqn{H_n})) \subset W_e (\seqn{H_n}).
    \end{equation*}
Furthermore, if $D(H_n) \cap D(H^*_n)$ is a core of $H^*_n$ for all $n$ then 
\begin{equation*}
     \text{\rm conv}(\sigess(\seqn{H_n}) \cup \sigess(\seqn{H^*_n})^* ) \subset W_e (\seqn{H_n}).
    \end{equation*}
\end{proposition}

\subsection{Enclosures for the Limiting Essential Spectrum}\label{subsec:discussion}

Throughout the remainder of the section, let $T_0$ and $s_n$ ($n \in \N$) be self-adjoint operators on $ \mathcal{H}$. 
Let $\gamma > 0$ and define the perturbed operators, as in the introduction, by 
\begin{equation}\label{eq:Tn-abstract}
    T_n = T_0 + i \gamma s_n.\qquad(n \in \N)
\end{equation}
Assume that $s_n \sto I$  and that $\norm{s_n} \leq C$ for some $C > 0$ independent of $n$. 
Define the limit operator by $T = T_0 + i \gamma$ as in the introduction. $T_n$ converges to $T$ in the strong sense, and in fact, as we shall show in the following proof, in the strong resolvent sense.  

\begin{corollary}\label{col:exactness} $\seqn{T_n}\text{ is spectrally exact for }T\text{ in }\C \bs \sbr*{\sigess(\seqn{T_n})\cup\sigess(\seqn{T^*_n})^* \cup \sigess(T)}$
\end{corollary}

\begin{proof}
 
 The fact that $T_n \srto T$ and $T_n^* = T_0 - i \gamma s_n \srto T_0 - i \gamma = T^*$ follows from an application of the resolvent identity, using $s_n \sto I$, the self-adjointness of $T_0$ and the uniform boundedness of the sequence of operators $\seqn{s_n}$.
 By Theorem \ref{thm:bogli-local}, $\sigpoll(\seqn{T_n}) \subset \sigess(\seqn{T_n}) \cup \sigess(\seqn{T^*_n})^*$ and 
 \begin{equation*}
\set{\lambda \in \sigma(T)  : \lambda \text{ isolated} ,\, \lambda \notin \sigess(\seqn{T_n})\cup\sigess(\seqn{T^*_n})^*} \subset \sigma(\seqn{T_n}).
 \end{equation*}
 The corollary follows from the fact that every element of $\sigdis(T) = \sigdis(T_0) + i \gamma$ is isolated since $T_0$ is self-adjoint \cite{edmunds2018spectral}.
\end{proof}

Since $D(T_n) = D(T^*_n) = D(T_0)$, Proposition \ref{prop:ess-num} implies that the set $\sigess(\seqn{T_n}) \cup \sigess(\seqn{T^*_n})^*$ is contained in the limiting essential numerical range $W_e(\seqn{T_n})$ and so
 $\seqn{T_n}$ is spectrally exact for $T$ in $\C \bs \sbr{W_e(\seqn{T_n}) \cup \sigess(T)}$. 
The limiting essential numerical range is typically easier to study than the limiting essential spectrum. 
For sequences of operators of the form \eqref{eq:Tn-abstract}, the limiting essential numerical range $W_e(\seqn{T_n})$ is contained in a strip.
To state this fact, we shall require the notion of extended essential spectrum.

\begin{definition}\label{def:sig-ess-extended}
 The \textit{extended essential spectrum} $\sigessh(H) \subset \sigess(H) \cup \set{\pm \infty}$ of a self-adjoint operator $H$ on $\mathcal{H}$ is defined as the union of $\sigess(H)$ with $+\infty$ and/or $-\infty$ if $H$ is unbounded from above and/or below respectively. 
\end{definition}
Throughout the remainder of the section, let 
\begin{equation}\label{eq:sp-and-sm}
s_- := \liminf_{n \to \infty} \inf_{u \in \mathcal{H}:\norm{u} = 1}\inner{s_nu,u}\quad\text{and}\quad s_+ := \limsup_{n \to \infty} \sup_{u \in \mathcal{H}:\norm{u} = 1}\inner{s_n u,u}.
\end{equation}
Then, for any $\epsilon > 0$ and sufficiently large $n$, $s_- - \epsilon \leq s_n \leq s_+ + \epsilon$. 
\begin{proposition}\label{prop:lim-num-est}
$W_e(\seqn{T_n}) \subset \sbr*{ \text{\rm conv}(\sigessh(T_0)) \bs \set{\pm \infty} } \times i\gamma[s_-,s_+]$
\end{proposition}
\begin{proof}
 Let $\lambda \in W_e(\seqn{T_n})$.
 Then there exist $I \subset \N$ infinite and $\seqn{u_n}_{n \in I} \subset D(T_0)$ such that $\norm{u_n} = 1$ for all $n \in I$, $u_n \wto 0$ and $\inner{(T_n - \lambda)u_n,u_n} \to 0$. 
 Taking the real part of the inner product, we have $\inner{(T_0 - \Re(\lambda))u_n,u_n} \to 0$ which implies that 
 \begin{equation*}
  \Re(\lambda) \in W_e(T_0) =\text{\rm conv}(\sigessh(T_0)) \bs \set{\pm \infty}
 \end{equation*}
where we used \cite[Theorem 3.8]{EssNumRan2020} in the equality. Finally, $\Im \inner{(T_n - \lambda)u_n,u_n} \to 0$ implies that $\Im(\lambda) = \gamma \inner{s_n u_n, u_n} + o(1) \in \gamma [s_-,s_+]$. 
\end{proof}

\subsection{Main Abstract Results}\label{subsec:main}

In the main result of this section, Theorem \ref{thm:encl}, we shall prove non-convex enclosures for the limiting essential spectrum $\sigess(\seqn{T_n})$ that complement the enclosure provided by the limiting essential numerical range.
We shall require additional assumptions on the perturbing operators $\seqn{s_n}$.
In part (a) of the theorem, we simply require that $s_n$ are projection operators.
An interesting feature of the enclosure of part (a) is that it is independent of the perturbing operators $\seqn{s_n}$, depending only on $\sigess(T_0)$ and $\gamma$.
The hypothesis for part (b) of the theorem, Assumption \ref{ass:loc-flat}, is given below.
An example of a class of perturbations for Schr\"odinger operators satisfying this assumption is provided in Example \ref{ex:loc-flat}. 
The enclosures are illustrated in Figure \ref{fig:encl}.

\begin{figure}[t]
\centering
\includegraphics[width=\linewidth]{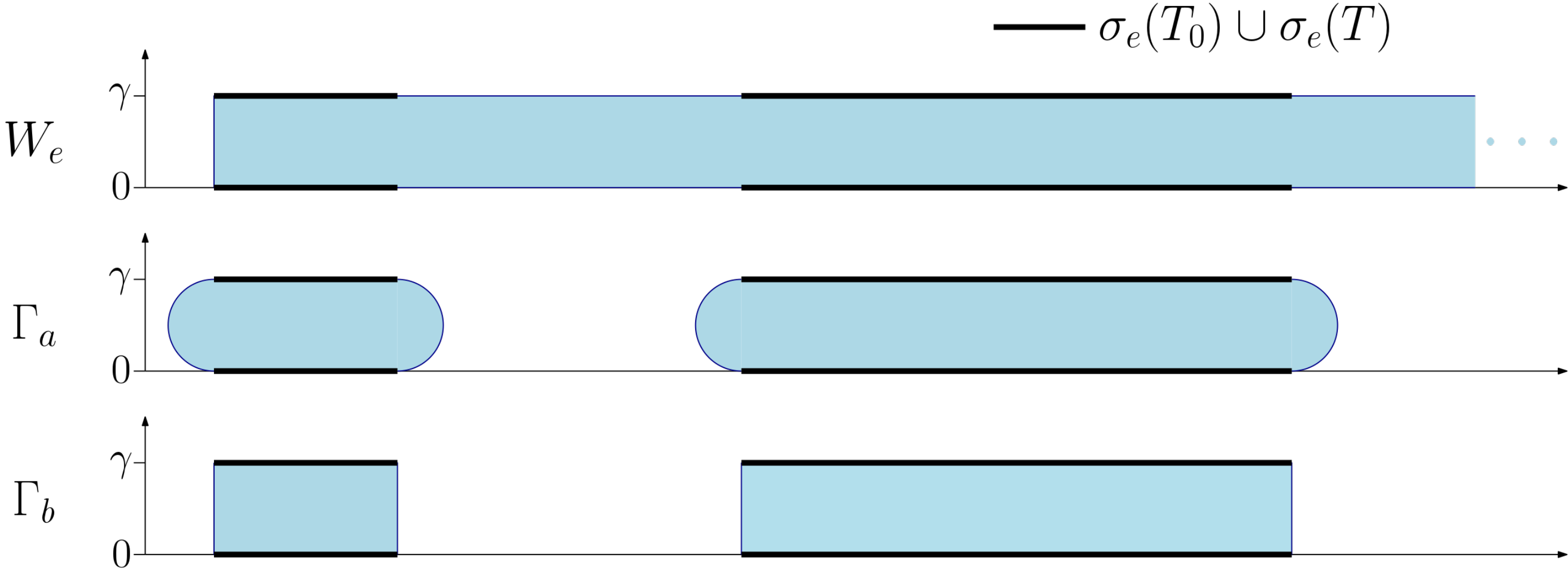}
\caption{Illustration of various enclosures for the limiting essential spectrum: the limiting essential numerical range $W_e = W_e(\seqn{T_n})$, the enclosure $\Gamma_a = \Gamma_a(\sigess(T_0), \gamma)$ of Theorem \ref{thm:encl} (a) and the enclosure $\Gamma_b = \Gamma_b(\sigess(T_0), \gamma, s_\pm)$ of Theorem \ref{thm:encl} (b). The illustration assumes that $T_0$ is unbounded only from above, $(s_-,s_+)=(0,1)$ and that the plotted region shows the smallest two spectral bands.}
\label{fig:encl}
\end{figure}

\begin{lemma}\label{lem:kato}
 Let $H$ be a self-adjoint operator on $\mathcal H$.
 If for some $\eta \in \R$ and $\epsilon > 0$ there exists a sequence $\seqn{u_n} \subset D(H)$ with $\norm{u_n} = 1$ for all $n$, $u_n \wto 0$ and  $\norm{(H - \eta) u_n} \to \epsilon$ then
\begin{equation*}
 \text{\rm dist}(\eta, \sigess(H)) \leq \epsilon.
\end{equation*}
\end{lemma}

\begin{proof}

 For any $\delta > 0$ there exists $N_\delta \in \N$ such that $\norm{(H - \eta)u_n} < (\epsilon + \delta) \norm{u_n}$ for all $n \geq N_\delta$.
 $(u_n)_{n \geq N_\delta}$ is a non-compact, bounded sequence so by \cite[Chapter I, Theorem 10]{glazman1965} the interval $(\eta - (\epsilon + \delta),\eta + (\epsilon + \delta))$ contains an infinite number of points in $\sigma(H)$.
Taking the limit $\delta \to 0$ shows that the interval $[\eta - \epsilon ,\eta + \epsilon ]$ contains an infinite number of points in $\sigma(H)$.
Finally, $[\eta - \epsilon ,\eta + \epsilon ]$ must contain a point of $\sigess(H)$ because any limit point of $\sigdis(H)$ is in $\sigess(H)$.
\end{proof}

\begin{assumption}\label{ass:loc-flat}
    If $\seqn{u_n} \subset D(T_0)$ is bounded in $\mathcal{H}$ with $\seqn{T_0 u_n}$  bounded in $\mathcal{H}$ then 
   \begin{equation*}
    \inner{s_n u_n, T_0 u_n} - \inner{T_0 u_n, s_n u_n} \to 0 \text{ \rm as }n \to \infty.
   \end{equation*}
\end{assumption}

\begin{theorem}\label{thm:encl}
 \begin{enumerate}[label=\rm{(\alph*)},wide, labelindent=0pt]
  \item  If $s_n$ is a projection operator, that is $s_n^2 = s_n$, for all $n$ then $\sigess(\seqn{T_n}) \cup \sigess(\seqn{T^*_n})^* \subset \Gamma_a = \Gamma_a(\sigess(T_0),\gamma)$
  where $\Gamma_a$ is defined by \eqref{eq:Gamma-a}.
  \item If Assumption \ref{ass:loc-flat} holds then $\sigess(\seqn{T_n}) \cup \sigess(\seqn{T^*_n})^* \subset \Gamma_b = \Gamma_b(\sigess(T_0),\gamma,s_\pm)$
   where
   \begin{equation}
        \Gamma_b := \sigess(T_0) \times i \gamma [s_-,s_+]. 
   \end{equation}
 \end{enumerate}
\end{theorem}

\begin{proof}
We will only prove that $\sigess(\seqn{T_n}) \subset \Gamma_a\text{ or }\Gamma_b$ - the proof that $\sigess(\seqn{T^*_n})^* \subset \Gamma_a\text{ or }\Gamma_b$ is similar since $T^*_n = T_0 - i \gamma s_n$.

Let $\lambda \in \sigess(\seqn{T_n})$.
Then there exist $I \subset \N$ infinite and $\seqn{u_n}_{n \in I} \subset D(T_0)$ with $\norm{u_n} = 1$ for all $n \in I$, $u_n \wto 0$ and 
$\norm{(T_n - \lambda)u_n} = o(1)$. 
By Cauchy-Schwarz, we have $\inner{(T_n - \lambda)u_n, u_n} = o(1)$, whose real and imaginary parts imply that 
\begin{equation}\label{eqpr:encl-2}
 \inner{T_0 u_n, u_n } = \Re(\lambda) + o(1) \quad \text{and}\quad \gamma \inner{s_n u_n, u_n } = \Im(\lambda) + o(1).
\end{equation}

Since both $\seqn{T_n u_n}$ and $\seqn{s_n u_n}$ are bounded in $\mathcal H$, $\seqn{T_0 u_n}$ must be bounded in $\mathcal H$.
Hence by Cauchy-Schwarz we have $\inner{(T_n - \lambda) u_n, T_0 u_n} = o(1)$ whose real part implies that 
\begin{equation}\label{eqpr:encl-5}
 \norm{T_0 u_n}^2 - \Re(\lambda) \inner{T_0 u_n, u_n} - \gamma \Im \inner{s_n u_n, T_0 u_n} = o(1).  
\end{equation}
The first equation in \eqref{eqpr:encl-2} gives
\begin{equation*}
    \norm{(T_0 - \Re(\lambda)) u_n }^2 = \norm{T_0 u_n}^2 - \Re(\lambda) \inner{T_0 u_n, u_n} + o(1),
\end{equation*}
which, combined with \eqref{eqpr:encl-5}, yields, 
\begin{equation}\label{eqpr:encl-6}
    \norm{(T_0 - \Re(\lambda))u_n}^2 = \gamma \Im \inner{s_n u_n, T_0 u_n} + o(1).
\end{equation}

\begin{enumerate}[label=\rm{(\alph*)},wide, labelindent=0pt]
\item 
In this case, $\sigma(s_n) = \set{0,1}$ so $0 \leq s_n \leq 1$ for all $n$, and so by the second equation in \eqref{eqpr:encl-2},
\begin{equation}\label{eqpr:encl-7}
 \forall n \in I : \inner{s_n u_n, u_n} \in [0,1] \quad \Rightarrow \quad \Im(\lambda) \in [0,\gamma]. 
\end{equation}
Focusing now on $\Re (\lambda)$, Cauchy-Schwarz gives us $\inner{(T_n - \lambda) u_n, s_n u_n} = o(1)$, whose imaginary part combined with the hypothesis $s_n^2 = s_n$ and the second equation in \eqref{eqpr:encl-2} gives,
\begin{align}\label{eqpr:encl-9}
 \Im \inner{s_n u_n, T_0 u_n} & = \gamma \norm{s_n u_n}^2 - \Im(\lambda) \inner{s_n u_n, u_n} + o(1) \nonumber \\
 & = (\gamma - \Im(\lambda) )\frac{\Im(\lambda)}{\gamma} + o(1).
\end{align}
Combining \eqref{eqpr:encl-6} and \eqref{eqpr:encl-9}, we have 
\begin{equation}\label{eqpr:encl-10}
 \norm{(T_0 - \Re(\lambda))u_n} = \sqrt{(\gamma - \Im(\lambda))\Im(\lambda)} + o(1)
\end{equation}
which by Lemma \ref{lem:kato} implies that 
\begin{equation*}
    \text{\rm dist}(\Re(\lambda),\sigess(T_0)) \leq \sqrt{(\gamma - \Im(\lambda))\Im(\lambda)}.
\end{equation*}

\item
In this case, by the definitions of $s_\pm$ in \eqref{eq:sp-and-sm}, a similar reasoning as in \eqref{eqpr:encl-7} yields $\Im(\lambda) \in \gamma [s_-,s_+]$. Assumption \ref{ass:loc-flat} implies that 
\begin{equation*}
    \Im \inner{s_n u_n, T_0 u_n} = o(1) \quad \Rightarrow \quad \norm{(T_0 - \Re(\lambda))u_n} = o(1)
\end{equation*}
so $\seqn{u_n}$ is a singular sequence proving that $\Re(\lambda) \in \sigess(T_0)$.
\end{enumerate}
\end{proof}

\begin{remark}
It is interesting to note that Lemma \ref{lem:kato} is not required in case (b) of Theorem \ref{thm:encl}.
This is because Assumption \ref{ass:loc-flat} ensures that the following holds:
\begin{multline*}
    \seqn{u_n} \subset D(T_n) = D(T_0),\,\norm{u_n} = 1,\,u_n \wto 0,\,\norm{(T_n - \lambda)u_n} \to 0 \\
    \Rightarrow \seqn{u_n} \subset D(T_0),\,\norm{u_n} = 1,\,u_n \wto 0,\,\norm{(T_0 - \Re(\lambda)u_n} \to 0,
\end{multline*}
that is, if $\seqn{u_n}$ is a singular-type sequence for a point $\lambda$ in the limiting essential spectrum then $\seqn{u_n}$ is also a singular sequence for $\Re(\lambda) \in \sigess(T_0)$.
\end{remark}

\begin{example}\label{ex:loc-flat}
Suppose that $q \in L^1_{\text{loc}}(\Omega)$ is real-valued and bounded from below. 
Suppose that $T_0 = - \Delta + q$  on $L^2(\Omega)$ is endowed with Dirichlet boundary conditions (see, for example, \cite[Chapter VII, Theorem 1.4]{edmunds2018spectral}).
Then, $T_0$ is self-adjoint. 

Let $\varphi \in W^{1, \infty}(0,\infty)$ be real-valued and such that $\varphi(0) = 1$.
Let $\seqn{R_n} \subset \R_+$ be any sequence such that $R_n \to \infty$.
For any $n \in \N$, define multiplication operator $s_n$ on $L^2(\Omega)$ by 
\begin{equation}
    (s_n u)(x) = \varphi \br*{\frac{\ang{x}}{R_n}} u(x) \qquad(u \in L^2(\Omega),\,x \in \Omega)
\end{equation}
where $\ang{x} := (1 + |x|^2)^{\frac{1}{2}}$.

Then $s_n$ is uniformly bounded, $s_n \sto I$ and Assumption \ref{ass:loc-flat} is satisfied. 
\end{example}

\begin{proof}
Define $\varphi_n:\Omega \to \R$ by
\begin{equation*}
    \varphi_n(x) = \varphi \br*{\frac{\ang{x}}{R_n}}. \qquad ( x \in \Omega) 
\end{equation*}
\noindent \textit{Step 1} (Uniform boundedness). 
The uniform boundedness of the sequence of operators $\seqn{s_n}$ follows from the fact that, for all $u \in L^2(\Omega)$ and all $n \in \N$, 
\begin{equation*}
    \essinf_{t \in (0,\infty)} \varphi(t) \norm{u}^2 \leq \inner{s_n u, u} \leq \esssup_{t \in (0,\infty)} \varphi(t) \norm{u}^2.
\end{equation*}

\noindent \textit{Step 2} ($s_n \sto I$).
Let $u \in L^2(\Omega)$ and let $\seqn{X_n} \subset \R_+$ be any sequence such that $X_n \to \infty$ and $X_n = o(R_n)$.
For any $n \in \N$, 
\begin{equation}\label{eqpr:loc-flat-1}
    \norm{(s_n - I)u} \leq \norm{\varphi(\ang{\cdot}/ R_n) - 1}_{L^\infty(\Omega \cap B_{X_n}(0))} \norm{u} + \br*{\norm{s_n} + 1}\norm{u}_{L^2(\Omega \bs B_{X_n}(0))}.
\end{equation}
By Morrey's inequality, $\varphi$ is continuous, so, since $\varphi(0) = 1$, the first term on the right hand side of \eqref{eqpr:loc-flat-1} tends to zero as $n \to \infty$. 
The second term tends to zero because $u \in L^2(\Omega)$ and $\seqn{\norm{s_n}}$ is bounded. 

\noindent \textit{Step 3} (Assumption \ref{ass:loc-flat}).
Let $\seqn{u_n} \subset D(T_0)$ be any sequence which is bounded in $\mathcal H$ such that $\seqn{T_0 u_n}$ is bounded in $\mathcal{H}$. 
Then, 
\begin{align*}
    \inner{s_n u_n, T_0 u_n} - \inner{T_0 u_n, s_n u_n}  & = - \int_{\Omega} \varphi_n u_n \Delta (\overline{u}_n) + \int_{\Omega} \varphi_n \overline{u}_n \Delta (u_n)  \\
   & = \int_{\Omega} u_n \nabla \br*{ \varphi_n} \cdot \nabla( \overline{u}_n) - \int_{\Omega} \overline{u}_n  \nabla \br*{\varphi_n} \cdot \nabla( u_n). 
\end{align*}
The second equality above holds by integration by parts and the product rule since $T_0$ is endowed with Dirichlet  boundary conditions. 
Hence we have, 
\begin{equation}\label{eqpr:loc-flat-2}
\abs*{\inner{s_n u_n, T_0 u_n} - \inner{T_0 u_n, s_n u_n}} \leq 2 \norm{\nabla \varphi_n}_{L^\infty(\Omega)} \norm{\nabla u_n} \norm{u_n}.
\end{equation}
By the chain rule and the fact that $\varphi \in W^{1,\infty}(0,\infty)$, $\norm{\nabla \varphi_n}_{L^\infty(\Omega)} \to 0$ as $n \to \infty$. 
$\seqn{u_n}$ is bounded in $\mathcal H$ by hypothesis. $\seqn{\nabla u_n}$ can be seen to be bounded in $\mathcal H$ by applying integration by parts to $\inner{T_0 u_n,u_n}$, using the hypotheses that $\seqn{\norm{T_0 u_n}}$ is bounded and that $q$ is bounded below.
The right hand side of \eqref{eqpr:loc-flat-2} tends to zero as $n \to \infty$ hence Assumption \ref{ass:loc-flat} is satisfied.
\end{proof}

\section{Second Order Operators on the Half-Line}\label{sec:spec-incl}

Consider the differential expression 
\begin{equation*}
    \tilde{T}_0 u = \frac{1}{r}\br*{ - (pu')' + qu }\quad \text{on} \quad [0,\infty) 
\end{equation*}
where $p$, $q$ and $r$ are functions on $[0,\infty)$ satisfying the minimal hypotheses: $p$ and $q$ are complex in general, $r > 0$, $p \neq 0$ and $q,1/p,r \in L^1_{\text{loc}}[0,\infty)$. These assumptions on $p$, $q$ and $r$ ensure that for any $\lambda, u_1, u_2 \in \C$ there exists a unique solution $u$ to the initial value problem 
\begin{equation*}
    \tilde{T_0} u = \lambda u \quad \text{on} \quad [0,\infty),\, u(0) = u_1,\, pu'(0) = u_2
\end{equation*}
such that $u,pu'\in AC _{\text{loc}}[0,\infty)$. The solution space of $\tilde{T_0} u = \lambda u$ on $[0, \infty)$ is therefore a two-dimensional complex vector space. 

Consider a Sturm-Liouville operator $T_0$ on the weighted Lebesgue space $L^2_r(0,\infty)$, endowed with a complex mixed boundary condition at 0,
\begin{equation}\label{eq:BC-defn}
    BC[u] := \cos (\eta ) u(0) - \sin ( \eta ) pu'(0) = 0 
\end{equation}
for some $\eta \in \C$. $T_0$ is defined by 
\begin{align}\label{eq:T0-defn}
\begin{split}
    T_0 u & = \tilde{T}_0 u \\
    D(T_0) & = \set{u \in L^2_r(0,\infty) : u, pu' \in AC_{\text{loc}}[0,\infty), \tilde{T}_0 u \in L^2_r(0, \infty), BC[u] = 0}. 
\end{split}
\end{align}

Fix $\gamma \in \C \bs \set{0}$. Define the perturbed operators by 
\begin{equation}\label{eq:TR-defn}
    T_R u = T_0 u + i \gamma \chi_{[0,R]} u,\, D(T_R) = D(T_0) \qquad (R \in \R_+)
\end{equation}
and define the limit operator by $\T = T_0 + i \gamma$.

Next, we introduce the main hypotheses of this section, which we will later assume holds throughout the section.
The assumption ensures that for any $\lambda \in \C \bs \sigess(T_0)$, one solution of $\tilde{T}_0 u = \lambda u$ is exponentially decaying and the other is exponentially growing.
\begin{assumption}\label{ass:exp}
There exists $k: \C \backslash \sigess (T_0) \to \C$, $\tilde{\psi}_\pm: [0,\infty) \times \C \bs \sigess (T_0) \to \C $ and $\tilde{\psi}^d_\pm: [0,\infty) \times \C \bs \sigess (T_0) \to \C $ such that:
\begin{enumerate}[label=(\roman*)]
    \item $k$ is analytic and satisfies $\Im k > 0$.
    \item $\tilde{\psi}_{\pm}(x,\cdot)$ and $\tilde{\psi}^d_{\pm}(x,\cdot)$ are analytic for all $x$ and satisfy
    \begin{equation}\label{eq:psi-t-Linf}
        \norm{\tilde{\psi}_\pm (\cdot,z)}_{L^\infty (0, \infty)} < \infty,\quad\norm{\tilde{\psi}^d_\pm (\cdot,z)}_{L^\infty (0, \infty)} < \infty
    \end{equation}
    for all $z$.
    \item The solution space of $\tilde{T}_0 u = z u$ is spanned by $\psi_\pm(\cdot,z)$, where,
    \begin{align}\label{eq:psi-pm-defn}
    \begin{split}
        \psi_\pm(x,z) & := e^{\pm i k(z) x}\tilde{\psi}_\pm(x,z) \\
        \psi'_\pm(x,z) & := e^{\pm i k(z) x}\tilde{\psi}^d_\pm(x,z).
    \end{split}
    \end{align}
    \end{enumerate}
\end{assumption}

\begin{remark}[See \cite{brown1999spectrum}]
    The conditions of Assumption \ref{ass:exp} do not exclude a situation in which $\sigma(T_0) = \sigess(T_0) = \C$. A sufficient condition to ensure that this does not occur is that 
    \begin{equation*}
        \overline{\text{co}}\set*{\frac{q(x)}{r(x)} + y p(x) : x,y \in [0,\infty)} \neq \C,
    \end{equation*}
    where $\overline{\text{co}}$ denotes the closed convex hull, and that $\tilde{T}_0$ is in Sims case I. 
    
\end{remark}

\begin{example}[Schr\"odinger operators with $L^1$ potentials]\label{ex:Lev-1}
Consider the case $p = r = 1$ with $q \in L^1(0,\infty)$. Then,
\begin{equation*}
    \sigess(T_0) = [0,\infty).
\end{equation*}
By the Levinson asymptotic theorem \cite[Theorem 1.3.1]{eastham1989asymptotic}, for any $ z \in \C \bs \set{0}$, the solution space of $\tilde{T}_0 u =  z u$ is spanned  by $\psi_\pm(\cdot, z)$, where 
\begin{align}
    \psi_\pm(x, z) & = e^{\pm i \sqrt{ z} x} \br*{1 + E_\pm(x, z)} \\
    \psi_\pm'(x,z) & = \pm i \sqrt{z} e^{\pm i \sqrt{z} x} \br*{1 + E^d_\pm(x,z)} 
\end{align}
and 
\begin{equation*}
    |E_\pm(x,z)|,|E^d_\pm(x,z)| \to 0 \text{ as }x \to \infty. 
\end{equation*}
\begin{enumerate}[label=\rm{(\roman*)},wide, labelindent=0pt]
\item
$k(z) := \sqrt{z}$ is analytic and satisfies $\Im k > 0$ on $\C \bs \sigess(T_0) = \C \bs [0,\infty)$.
\item
$\tilde{\psi}_\pm(x,z):=1 + E_\pm(x,z)$ and $\tilde{\psi}^d_\pm(x,z):= \pm i \sqrt{z}\br{1 + E^d_\pm(x,z)}$ are bounded in $x$ for any fixed $z \in \C \bs \set{0}$. For any $x$, $\psi_\pm(x,\cdot)$ and $\psi^d_\pm(x,\cdot)$ are analytic on $\C \bs [0,\infty)$ so  $\tilde{\psi}_\pm(x,\cdot)$ and $\tilde{\psi}^d_\pm(x,\cdot)$ are analytic on $\C \bs [0,\infty)$.
\end{enumerate}
Consequently, Assumption \ref{ass:exp} is satisfied in this case.
\end{example}

\begin{example}[Eventually periodic Schr\"odinger operators]\label{ex:Floq-1}
Consider the case $p = r = 1$ with $q$ eventually real periodic, that is, there exists $a>0$ and $X \geq 0$ such that $q|_{[X,\infty)}$ is real-valued and $a$-periodic. Below, we briefly review some Floquet theory and show that the conditions of Assumption \ref{ass:exp} are met in this case.  See, for example, \cite{eastham1973periodic} for a detailed exposition of Floquet theory.

For any $z \in \C$, let $\phi_1(\cdot,z)$ and $\phi_2(\cdot,z)$ be the solutions of the Schr\"odinger equation $ - \phi'' + q \phi = z \phi$ on $[0,\infty)$, subject to the boundary conditions
\begin{equation}\label{eq:phi-1-2}
    \phi_1(X,z) = 1,\, \phi_1'(X,z) = 0\text{ and }\phi_2(X,z) = 0,\, \phi_2'(X,z) = 1.
\end{equation}
The \textit{discriminant} is defined by
\begin{equation}
    D(z) = \phi_1(X + a,z) + \phi_2'(X + a,z).
\end{equation}
The essential spectrum of $T_0$ is 
\begin{equation}
    \sigess(T_0) = \set{z \in \R: |D(z)| \leq 2}. 
\end{equation}
The \textit{Floquet multipliers} $\rho_\pm$ are defined by 
\begin{equation}\label{eq:Floq-mult}
    \rho_\pm(z) = \frac{1}{2}\br*{ D(z) \pm i \sqrt{4 - D(z)^2}}.
\end{equation}
Note that $\rho_\pm$ have branch cuts along $\sigess(T_0)$, $|\rho_+(z)| < 1$ for all $z \in \C \bs \sigess(T_0)$ and $\rho_+(z)\rho_-(z) = 1$.
Define $k$ by 
\begin{equation}
    k(z) = -\frac{i}{a}\ln\br*{\rho_+(z)}.
\end{equation}
In this setting, $k$ is referred to as the \textit{Floquet exponent}. $k$ is analytic and satisfies $\Im k > 0$ on $\C \bs \sigess(T_0)$ hence satisfies Assumption (i).  

Define the \textit{Floquet solutions} $\psi_\pm$ by 
\begin{equation}\label{eq:Floq-soln-eigenvect}
    \psi_\pm(x,z) = - \phi_2(X+a,z) \phi_1(x,z) +  (\phi_1(X+a,z) - \rho_\pm(z)) \phi_2(x,z)
\end{equation}
for any $x \in [0,\infty)$ and $z \in \C$. 
$\psi_\pm(\cdot,z)$ span the solution space of $\tilde{T}_0 u = z u$ and satisfy
\begin{align}\label{eq:Floq-soln-prop}
\begin{split}
    \psi_\pm(x_0 + n a, z) & = e^{\pm i k(z) n a} \psi_\pm(x_0,z)  \\
    \psi'_\pm(x_0 + n a, z) & = e^{\pm i k(z) n a} \psi'_\pm(x_0,z)  
\end{split}
\end{align}
for any $x_0 \in [X,X+a)$ and $n \in \N$. For any $x$, the Floquet solutions $\psi_\pm(x,\cdot)$ and $\psi'_\pm(x,\cdot)$ are analytic on $\C \bs \sigess(T_0)$.
Define the \textit{band-ends} $B_{\text{ends}}$ by
\begin{equation}\label{eq:B-ends}
    B_{\text{ends}} = \set{\lambda \in \C: |D(\lambda)| = 2}.
\end{equation}
For any $z_0 \in \sigess(T_0) \bs B_{\text{ends}}$,  $\rho_\pm$ and $k$ can be analytically continued into an open neighbourhood of $z_0$ in $\C$, hence for any $x \in [0,\infty)$, $\psi_\pm(x,\cdot)$ and $\psi'_\pm(x,\cdot)$ can be analytically continued into an open neighbourhood of $z_0$.

Finally, Assumption \ref{ass:exp} can be satisfied by setting 
\begin{equation}\label{eq:Floq-psit}
    \tilde{\psi}_\pm(x,z) = \begin{cases}
    e^{\mp i k(z)x} \psi_\pm(x,z) & \text{if }x \in [0,X) \\
    e^{\mp i k(z)x_0(x)} \psi_\pm(x_0(x),z) & \text{if }x \in [X,\infty)
    \end{cases}
\end{equation}
and 
\begin{equation}\label{eq:Floq-psit-d}
    \tilde{\psi}^d_\pm(x,z) = \begin{cases}
    e^{\mp i k(z)x} \psi'_\pm(x,z) & \text{if }x \in [0,X) \\
    e^{\mp i k(z)x_0(x)} \psi'_\pm(x_0(x),z) & \text{if }x \in [X,\infty)
    \end{cases}
\end{equation}
where $x_0(x) := X + (x - X)\text{mod } a $.

\end{example}

Throughout the remainder of the section, let 
\begin{equation}\label{eq:S-defn}
    S := \sigess(T_0)\cup(i \gamma + \sigess(T_0))    
\end{equation}
and suppose that the conditions of Assumption \ref{ass:exp} are satisfied.
Also, let $\seqn{R_n} \subset \R_+$ be any sequence such that $R_n\to \infty$ as $n \to \infty$.
Recall that $BC$ denotes the boundary condition functional defined by equation \eqref{eq:BC-defn}.

\begin{lemma}\label{lem:fR}
    $\lambda \in \C \bs S$ is an eigenvalue of $T_R$ if and only if 
    \begin{equation*}
        f_R(\lambda) := \alpha_+(R,\lambda)e^{i k(\lambda - i \gamma) R} + \alpha_-(R,\lambda) e^{ - i k(\lambda - i \gamma) R} = 0
    \end{equation*}
    where
    \begin{equation*}
        \alpha_+(R,\lambda) := BC[\psi_-(\cdot,\lambda - i \gamma)] \br*{ \tpsip(R,\lambda - i \gamma)\tpsipd(R,\lambda) - \tpsipd(R,\lambda - i \gamma) \tpsip(R,\lambda)}
    \end{equation*}
    and
    \begin{equation*}
        \alpha_-(R,\lambda) := BC[\psi_+(\cdot,\lambda - i \gamma)] \br*{ \tpsip(R,\lambda)\tpsimd(R,\lambda - i \gamma) - \tpsipd(R,\lambda) \tpsim(R,\lambda - i \gamma)}.
    \end{equation*}
    Furthermore, $f_R$ is analytic on $\C \bs S$.
\end{lemma}

\begin{proof}
Let $\lambda \in \C \bs S $ and $R > 0$. $\lambda$ is an eigenvalue of $T_R$ if and only if there exists a solution to the problem  
\begin{equation}\label{eqpr:fR-1}
    (\tilde{T}_0 + i \gamma \chi_{[0,R]})u = \lambda u,\, BC[u] = 0,\,u \in L^2_r(0,\infty)
\end{equation}
on $[0,\infty)$.
Any solution to \eqref{eqpr:fR-1} on $[0,R]$ must be of the form $C_1 u_1(\cdot,\lambda)$, where $u_1$ is defined by 
\begin{equation*}
    u_1(x,\lambda) = BC[\psi_-(\cdot,\lambda - i \gamma)] \psip(x,\lambda - i \gamma) - BC[\psi_+(\cdot,\lambda - i \gamma)] \psim(x,\lambda - i \gamma)
\end{equation*}
and $C_1 \in \C$ is independent of $x$. Any solution to \eqref{eqpr:fR-1} on $[R,\infty)$ must be of the form $C_2 \psip(x,\lambda)$, where $C_2 \in \C$ is independent of $x$. 
Hence $\lambda$ is an eigenvalue if and only if there exists $C_1,C_2 \in \C \bs \set{0}$ independent of $x$ such that the function
\begin{equation*}
    x \mapsto \begin{cases}
    C_1 u_1(x,\lambda) & \text{if }x \in [0,R) \\
    C_2 \psi_+(x,\lambda) & \text{if }x \in [R, \infty)
    \end{cases}
\end{equation*}
is absolutely continuous. 
This holds if and only if
\begin{equation*}
    u_1(R, \lambda) \psi_+'(R, \lambda) - u_1'(R, \lambda) \psi_+(R, \lambda ) = 0 
\end{equation*}
which holds if and only if the following quantity is zero
\begin{multline*}
    \br*{  BC[\psi_-(\cdot,\lambda - i \gamma)] \psip(R,\lambda - i \gamma) - BC[\psi_+(\cdot,\lambda - i \gamma)] \psim(R,\lambda - i \gamma)} \tpsipd(R, \lambda)  \\
    -  \br*{  BC[\psi_-(\cdot,\lambda - i \gamma)] \psip'(R,\lambda - i \gamma) - BC[\psi_+(\cdot,\lambda - i \gamma)] \psim'(R,\lambda - i \gamma)} \tpsip(R, \lambda) 
\end{multline*}
which in turn is equivalent to $f_R(\lambda) = 0$.
The analyticity claim follows from Assumptions \ref{ass:exp} (i) and (ii).
\end{proof}

In the regions of the complex plane for which $\alpha_-(R,\cdot)$ becomes small for large $R$, we are unable to prove the spectral pollution and spectral inclusion results of Theorems \ref{thm:exp-conv} and \ref{thm:spec-poll}. We now define a subset of the complex plane capturing such regions. 

\begin{definition}
Define subset $\Sp$ of $\C$ by 
\begin{equation}\label{eq:Sp-defn}
    \Sp = \set*{ z \in \C \bs S : \liminf_{n \to \infty} |\Lambda(R_n,z)| = 0 }
\end{equation}
where the function $\Lambda: [0,\infty)\times \C \bs S \to \C$ is defined by
\begin{equation}\label{eq:Lam-defn}
    \Lambda(R,\lambda) = \tpsip(R, \lambda) \tpsimd(R,\lambda - i \gamma) - \tpsipd(R,\lambda) \tpsim(R,\lambda - i \gamma).
\end{equation}
\end{definition}
Note that with the above definition of $\Lambda$, we have  
\begin{equation*}
    \alpha_-(R,\lambda) = BC[\psi_+(\cdot,\lambda - i \gamma)]\Lambda(R,\lambda)
\end{equation*}
and that the zeros of $\lambda \mapsto BC[\psi_+(\cdot, \lambda - i \gamma)]$ are exactly the eigenvalues of the limit operator $\T = T_0 + i \gamma$.

The set $S \cup \Sp$ plays a similar role in this section as the limiting essential spectrum did in Section \ref{sec:spec-poll}.
We shall show in Theorems \ref{thm:exp-conv} and \ref{thm:spec-poll} that there is no spectral pollution for $\seqn{T_{R_n}}$ with respect to $\T$ outside of $S \cup \Sp$ and that eigenvalues of $\T$ lying outside of $S \cup \Sp$ are approximated (with exponentially small error) by the eigenvalues of $T_{R_n}$.

\begin{proposition}\label{prop:S-Sp-closed}
    $S \cup \Sp$ is a closed subset of $\C$.
\end{proposition}

\begin{proof}
By Assumption \ref{ass:exp} (ii), $\Lambda(R_n,\cdot)$ is analytic for all $n$ and $\Lambda(\cdot,z)$ is bounded for all $z$. Let $\lambda$ be a limit point of $S\cup \Sp$. The desired lemma holds if and only if $\lambda$ lies in either $S$ or in $\Sp$. If $\lambda$ is a limit point of $S$ then $\lambda \in S$ since $S$ is closed. In the only other case, $\lambda$ is a limit point of $\Sp$ so there exists $\seqk{\lambda_k} \subset \Sp$ such that $\lambda_k \to \lambda $ as $k \to \infty$. Since $\liminf_{n \to \infty}|\Lambda(R_n,\lambda_k)| = 0$ for all $k$, there exists a subsequence $\seqk{R_{n_k}}$ such that $|\Lambda(R_{n_k},\lambda_k)| \to 0$ as $k \to \infty$. Let $\epsilon > 0$ be small enough so that $\overline{B_\epsilon (\lambda)} \subseteq \C \bs S$. 
Since the magnitude of $\Lambda(R,z)$ is bounded above uniformly for all $R>0$ and all $z \in \overline{B_\epsilon(\lambda)}$, by Cauchy's formula, 
\begin{equation}\label{eq:power-method}
    \Lambda(R_{n_k},\lambda) - \Lambda(R_{n_k},\lambda_k) = \frac{1}{2 \pi i}\oint_{\partial B_\epsilon(\lambda)} \frac{\lambda_k  - \lambda}{(z - \lambda)(z - \lambda_k)}\Lambda(R_{n_k},z) \d z \to 0
\end{equation}
as $k \to \infty$.
Finally, 
\begin{equation*}
    |\Lambda(R_{n_k},\lambda)| \leq |\Lambda(R_{n_k},\lambda_k)| + |\Lambda(R_{n_k},\lambda) - \Lambda(R_{n_k},\lambda_k)| \to 0 \text{ as }k \to \infty
\end{equation*}
so $\lambda \in \Sp$, completing the proof.
\end{proof}

\begin{corollary}\label{col:S-Sp-nhood}
 For any $\lambda \in \C \bs \br{S \cup \Sp}$ there exists a bounded, open neighbourhood $U$ of $\lambda$ with $\overline{U} \subset \C \bs S$ and $|\Lambda(R_n,z)| \geq C$ for all $z \in U$ and $n \geq N_0$, where $C,N_0 > 0$ are some constants independent of $n$ and $z$.
\end{corollary}

\begin{proof}
 Let $\lambda \in \C \bs (S \cup \Sp)$. $\C \bs (S \cup \Sp)$ is an open subset of $\C$ so there exists a bounded open neighbourhood $U$ of $\lambda$ such that $\overline{U} \subset \C \bs (S \cup \Sp)$. Suppose that the desired result does not hold with this choice for $U$.
Then there exists a subsequence $\seqk{R_{n_k}}$ and a sequence $\seqk{z_k} \subset U$ such that $|\Lambda(R_{n_k}, z_k)| \to 0$ as $k\to \infty$.  Since $\overline{U}$ is compact, there exists $z \in \C \bs (S \cup \Sp)$ such that $z_k \to z$. By the arguments in (a), $\liminf_{n\to\infty} |\Lambda(R_n,z)| = 0 $, which is the desired contradiction.
\end{proof}

Next, we prove the main results of this section, regarding spectral inclusion and pollution for the operators $T_R$ defined by equation \eqref{eq:TR-defn} such that $T_0$ satisfies Assumption \ref{ass:exp}. 
Recall, also, that $S$ is defined by equation \eqref{eq:S-defn}, $\Sp$ is defined by \eqref{eq:Sp-defn} and $\seqn{R_n} \subset \R_+$ is an arbitrary sequence such that $R_n \to \infty$.

\begin{theorem}\label{thm:exp-conv}
Let $\mu$ be an eigenvalue of $T_0$ and assume that $\mu + i \gamma \notin S \cup \Sp$. Then there exists eigenvalues $\lambda_n$ of $T_{R_n}$ $(n \in \N)$ and constants $C_0 = C_0(T_0,\gamma,\mu) > 0$ and $\beta = \beta(T_0,\gamma,\mu) > 0$  such that 
\begin{equation}\label{eq:exp-conv-ineq}
    |\lambda_{n} - (\mu + i \gamma)| \leq C_0 e^{- \beta R_n}
\end{equation}
for all large enough $n$. 
\end{theorem}

\begin{proof}
Let $ C, C_1, C_2, C_3,N_0 > 0$ denote constants independent of $\lambda$ and $n$, where $C$ may change from line to line.

Since $\mu$ is an eigenvalue of $T_0$, $\mu + i \gamma$ is a zero of the analytic function 
\begin{equation*}
    \lambda \mapsto \tilde{f}(\lambda) := BC[\psi_+(\cdot,\lambda - i \gamma)].
\end{equation*}
Since it is assumed that $\mu + i \gamma \notin S \cup \Sp$, Corollary \ref{col:S-Sp-nhood} guarantees the existence of an open neighbourhood $U$ of $\mu + i \gamma$ such that $\overline U \subseteq \C \bs S$  and $|\Lambda(R_n,\lambda)| \geq C$ ($\lambda \in U,\,n \geq N_0$) for some sufficiently large $N_0 \in \N$ . 
For $n \geq N_0$, $\lambda \in U$ is an eigenvalue of $T_{R_n}$ if and only if 
\begin{equation*}
    \tilde{f}_n(\lambda) := e^{i k(\lambda - i \gamma) R_n} \frac{f_{R_n}(\lambda)}{\Lambda(R_n,\lambda)} = 0.
\end{equation*}
Since $\overline{U} \in \C \bs S$, Assumption \ref{ass:exp} guarantees that $|\alpha_+(R_n,\lambda)| \leq C$ $(\lambda \in U, n \in \N)$ and $\Im k(\lambda - i \gamma) \geq C$ $(\lambda \in U)$.
Combined with the bound below for $\Lambda$, this implies that
\begin{equation}\label{eqpr:eig-1}
    |\tilde{f}_n(\lambda) - \tilde{f}(\lambda)| = \abs*{e^{2 i k(\lambda - i \gamma)R_n} \frac{\alpha_+(R_n,\lambda)}{\Lambda(R_n,\lambda)}} \leq C_1 e^{- C_2 R_n}\qquad (\lambda \in U, n \geq N_0) 
\end{equation}
for some $C_1, C_2 > 0$. 
Since $\tilde{f}$ is analytic at $\mu + i \gamma$, there exists $\epsilon > 0$ such that
\begin{equation}\label{eqpr:eig-2}
    |\tilde{f}(\lambda)| \geq C_3 |\lambda - (\mu + i \gamma)|^\nu \qquad (\lambda \in B_\epsilon (\mu + i \gamma))
\end{equation}
for some $C_3 > 0$.
Here, $\nu$ is the algebraic multiplicity of the eigenvalue $\mu$ of $T_0$, that is, the multiplicity of the zero $\mu$ of the analytic function $z \mapsto BC[\psi_+(\cdot,z)]$. 
Let $C_0 = (2 C_1/C_3)^{1/\nu}$ and $\beta = C_2 / \nu$.
Make $N_0 \in \N$ large enough such that $C_0 e^{- \beta R_n} < \epsilon$ $(n \geq N_0)$.
Combining \eqref{eqpr:eig-1} and \eqref{eqpr:eig-2}, for all $n \geq N_0$ and all $\lambda \in \C$ with 
\begin{equation*}
    |\lambda - (\mu + i \gamma)| = C_0 e^{- \beta R_n}
\end{equation*}
we have 
\begin{equation*}
    |\tilde{f}_n(\lambda) - \tilde{f}(\lambda)| \leq \frac{1}{2} |\tilde{f}(\lambda)| < |\tilde{f}(\lambda)|.
\end{equation*}
By Rouch\'e's theorem, for all $n \geq N_0$ there exists a zero $ \lambda_n \in U$ of $\tilde{f}_n$ satisfying inequality \eqref{eq:exp-conv-ineq}.
\end{proof}

The next result concerns spectral pollution - the set of spectral pollution is defined by equation \eqref{eq:spec-poll-defn}.

\begin{theorem}\label{thm:spec-poll}
The set of spectral pollution of the sequence of operators $\seqn{T_{R_n}}$ with respect to the limit operator $\T = T_0 + i \gamma$ satisfies
\begin{equation*}
    \sigpoll(\seqn{T_{R_n}}) \subseteq \sigess(T_0)\cup \Sp.
\end{equation*}
\end{theorem}

\begin{proof}
Let $C > 0$ denote an arbitrary constant independent of $\lambda$ and $n$.

Let $\mu \in \C \bs (S \cup \Sp)$ and assume that $\mu $ is not an eigenvalue of $\T$.
Then $\mu$ is an arbitrary element of $\rho(T) \bs (\sigess(T_0) \cup \Sp)$. We aim to show that $\mu \notin \sigpoll(\seqn{T_{R_n}})$, for which it suffices to show that there exists a neighbourhood $U$ of $\mu$ such that $f_{R_n}$ has no zeros in $U$ for large enough $n$.

Since $\mu \notin \Sp$ and $BC[\psi_+(\cdot,\mu - i \gamma)] \neq 0$,
\begin{equation}\label{eqpr:spec-poll-3}
    |\alpha_-(R_n,\mu)| = |BC[\psi_+(\cdot,\mu - i \gamma)]\Lambda(R_n,\mu)| \geq C
\end{equation}
for large enough $n$.
Let $\epsilon > 0$ be small enough so that $\overline{B_\epsilon(\mu)} \subseteq \C \bs S$.
Then by Assumption \ref{ass:exp} we have 
\begin{equation}\label{eqpr:spec-poll-2}
     |\alpha_\pm(R_n,\lambda)| \leq C \text{ and } \Im k(\lambda - i \gamma) \geq C \qquad (\lambda \in B_\epsilon(\mu), n \in \N)
\end{equation}
Using Cauchy's integral formula as in \eqref{eq:power-method}, and making $\epsilon > 0$ small enough, we have
\begin{equation}\label{eqpr:spec-poll-1}
    |\alpha_\pm(R_n,\lambda) - \alpha_\pm(R_n,\mu)| \leq C | \lambda - \mu | \qquad (\lambda \in B_\epsilon(\mu),n \in \N).
\end{equation}
Define approximation $f_n^{(\mu)}$ to $f_{R_n}$ by 
\begin{equation*}
    f_n^{(\mu)}(\lambda) := \alpha_+(R_n,\mu)e^{i k(\lambda - i \gamma) R_n} + \alpha_-(R_n,\mu)e^{-i k(\lambda - i \gamma) R_n}.
\end{equation*}
By \eqref{eqpr:spec-poll-1} we have
\begin{equation*}
    |f_{R_n}(\lambda) - f_n^{(\mu)}(\lambda)| \leq C |\lambda - \mu|e^{\Im k(\lambda - i \gamma) R_n} \qquad (\lambda \in B_\epsilon(\mu),n \in \N).
\end{equation*} 
Using \eqref{eqpr:spec-poll-3} and \eqref{eqpr:spec-poll-2} we have  
\begin{align*}
    |e^{i k(\lambda - i \gamma) R_n} f_n^{(\mu)}(\lambda)| & \geq \abs*{|\alpha_-(R_n,\mu)| - |\alpha_+(R_n,\mu)|e^{- 2 \Im k(\lambda - i \gamma) R_n}} \\
    & \geq \frac{|\alpha_-(R_n,\mu)|}{2} \geq C \qquad (\lambda \in B_\epsilon(\mu))
\end{align*}
for large enough $n$.
Finally, making $\epsilon > 0$  small enough if necessary, we have 
\begin{equation*}
    |f_{R_n}(\lambda) - f_n^{(\mu)}(\lambda)| < |f_n^{(\mu)}(\lambda)| \qquad (\lambda \in B_\epsilon(\mu))
\end{equation*}
for large enough $n$. $f_{R_n}$ therefore has no zeros in $U := B_\epsilon(\mu)$ for large enough $n$, completing the proof.

\end{proof}

In the case of Schr\"odinger operators on $L^2(0,\infty)$ with $L^1$ potentials, described in Example \ref{ex:Lev-1}, $\Sp$ can be easily computed to be the empty set. 

\begin{example}[Schr\"odinger operators with $L^1$ potentials, continued]\label{ex:Lev-2}
Consider again the case $p = r = 1$ with $q \in L^1(0,\infty)$.
Then, using expression \eqref{eq:Lam-defn} for $\Lambda$ and the expressions for $\tilde{\psi}_\pm,\,\tilde{\psi}^d_\pm$ in Example \ref{ex:Lev-1} (ii), $\Lambda$ satisfies 
\begin{equation*}
    \Lambda(R,\lambda) \to - i \br*{\sqrt{\lambda - i \gamma} + \sqrt{\lambda}} \text{ as }R \to \infty
\end{equation*}
for any $\lambda \in \C \bs S $. Since $\sqrt{\lambda - i \gamma} \neq  - \sqrt{\lambda} $ for all $\lambda \in \C$ we have 
\begin{equation*}
    \Sp = \emptyset
\end{equation*}
for any $\seqn{R_n} \subset \R_+$ with $R_n \to \infty$ as $n \to \infty$.
\end{example}

For Schr\"odinger operators with eventually real periodic potentials, described in Example \ref{ex:Floq-1}, the computation of $\Sp$ is more involved.
\begin{example}[Eventually periodic Schr\"odinger operators, continued]\label{ex:Floq-2}
Consider again the case $p = r = 1$ with $q|_{[X,\infty)}$ real-valued and $a$-periodic for some $X \geq 0$ and $a > 0$.   
Assume that $\gamma > 0$ and let $R_n = x_0 + n a$ ($n \in \N$) for any fixed $x_0 \in [X, X+a)$. 

Using the expressions \eqref{eq:Floq-psit} and \eqref{eq:Floq-psit-d} for $\tilde{\psi}_\pm$ and $\tilde{\psi}^d_\pm$ as well as the definition of $\Sp$ in equation \eqref{eq:Sp-defn}, we infer that $\lambda \in \Sp$ if and only if 
\begin{equation}\label{eq:Floq-Sp-zeros}
    \psi_+(x_0,\lambda) \psi'_-(x_0,\lambda - i \gamma) - \psi'_+(x_0,\lambda) \psi_-(x_0,\lambda - i \gamma) = 0.
\end{equation}
$\psi_\pm(x_0,\cdot)$ and $\psi'_\pm(x_0,\cdot)$ are analytic on $\C \bs \sigess(T_0)$ and can be analytically continued into an open neighbourhood in $\C$ of any point in $\sigess(T_0) \bs B_{\text{ends}}$ (recall that $B_{\text{ends}}$ denotes the set of band-ends for the essential spectrum of $T_0$).
Consequently, $\Sp$ consists of isolated points in the complex plane that can only accumulate to the band-ends of either $T_0$ or $T$, that is, to the set $B_{\text{ends}}\cup(i \gamma + B_{\text{ends}})$.

Recall that $\sigess((T_{R_n}))$ denotes the limiting essential spectrum of the sequence of operators $\seqn{T_{R_n}}$.
$\Sp$ satisfies the inclusion
\begin{equation}\label{eq:Floq-Sp-incl-1}
    \Sp \subseteq \sigess(\seqn{T_{R_n}}).
\end{equation}
\end{example}

\begin{proof}[Proof of inclusion \eqref{eq:Floq-Sp-incl-1}]

Throughout the proof, $C>0$ denotes an arbitrary constant independent of $n$ .

By $\norm{\cdot}_{L^2}$ and $\norm{\cdot}_{L^\infty}$, we mean $\norm{\cdot}_{L^2(0,\infty)}$ and $\norm{\cdot}_{L^\infty(0,\infty)}$ respectively. 

Let $\lambda \in \Sp$. 
Then, using the property \eqref{eq:Floq-soln-prop} of the Floquet solutions, \eqref{eq:Floq-Sp-zeros} implies that,
\begin{equation}\label{eqpr:Floq-bigex-1}
    \psi_+(R_n,\lambda) \psi'_-(R_n,\lambda - i \gamma) - \psi'_+(R_n,\lambda) \psi_-(R_n,\lambda - i \gamma) = 0 
\end{equation}
for all $n$. 
\eqref{eqpr:Floq-bigex-1} ensures that there exists $C_{1,n},C_{2,n} \in \C \bs \set{0} $ independent of $x$ such that 
\begin{equation}\label{eqpr:step-1-un}
    u_n(x) := \begin{cases} C_{1,n} \psi_-(x,\lambda - i \gamma) & \text{if }x \in [0,R_n) \\ C_{2,n} \psi_+(x,\lambda) & \text{if }x \in [R_n,\infty)
    \end{cases}
\end{equation}
is absolutely continuous and solves the Schr\"odinger equation $\tilde{T}_{R_n} u = \lambda u$, where $\tilde{T}_{R_n}$ denotes the differential expression on $[0,\infty)$ corresponding to $T_{R_n}$.
Define 
\begin{equation*}
    v_n = \frac{\tilde{\chi}_n u_n}{\norm{\tilde{\chi}_n u_n}_{L^2}}. 
\end{equation*}
where $\tilde{\chi}_n(x) := \tilde{\chi}(x/R_n)$ and $\tilde{\chi}:[0,\infty) \to [0,1]$ is any smooth function such that $\tilde{\chi} = 0$ on $[0,\frac{1}{4}]$ and $\tilde{\chi} = 1$ on $[\frac{1}{2},\infty)$. 
Then $v_n \in D(T_0)=D(T_{R_n})$, $\norm{v_n}_{L^2} = 1$ and, since $\inner{v_n,\varphi}_{L^2} = 0$ for any $\varphi\in C^\infty_c[0,\infty)$ and any large enough $n$, $v_n \wto 0$ in $L^2(0,\infty)$. 

By unique continuation, 
\begin{equation*}
    \norm{\psi'_-(\cdot,\lambda - i \gamma)}_{L^2(I)} \leq C \norm{\psi_-(\cdot,\lambda - i \gamma)}_{L^2(I)} 
\end{equation*}
for $I = [0,X],\,[X,X+a]$ or $[X,x_0]$ so, using the property \eqref{eq:Floq-soln-prop} of the Floquet solutions
\begin{equation}\label{eq:step-1-1}
     \norm{\psi'_-(\cdot,\lambda - i \gamma)}^2_{L^2(0,R_n)} \leq C \norm{\psi_-(\cdot,\lambda - i \gamma)}^2_{L^2(0,R_n)}
\end{equation}
for all $n$. 
Also, noting that $\norm{\psi_-(\cdot,\lambda-i\gamma)}_{L^2(0,x)}$ is exponentially growing in $x$, we deduce that,
\begin{equation}\label{eq:step-1-2}
    \norm{u_n}_{L^2} \leq C \norm{u_n}_{L^2(\frac{1}{2}R_n,\infty)} \leq C \norm{\tilde{\chi}_n u_n}_{L^2}.    
\end{equation}
for all large enough $n$.

By the product rule,
\begin{align*}
    \norm{(T_{R_n} - \lambda)v_n}_{L^2} \leq \frac{1}{\norm{\tilde{\chi}_n u_n}_{L^2}} \sbr*{ \norm{\tilde{\chi}_n(\tilde{T}_{R_n} - \lambda)u_n}_{L^2} + 2 \norm{\tilde{\chi}'_n u'_n}_{L^2} + \norm{\tilde{\chi}''_n u_n}_{L^2} }.
\end{align*}
The first term in the square brackets above vanishes and $\tilde{\chi}^{(k)}_n$ are supported in $[0,R_n]$ with $\norm{\tilde{\chi}^{(k)}_n}_{L^\infty} \leq C / R_n^k$ so 
\begin{equation*}
    \norm{(T_{R_n} - \lambda)v_n}_{L^2} \leq C \frac{\norm{u_n}_{L^2}}{\norm{\tilde{\chi}_n u_n}_{L^2}} \sbr*{ \frac{1}{R_n}\frac{\norm{\psi'_-(\cdot,\lambda - i \gamma)}_{L^2(0,R_n)}}{\norm{\psi_-(\cdot,\lambda - i \gamma)}_{L^2(0,R_n)}} + \frac{1}{R_n^2} } \to 0\text{ as }n \to \infty.
\end{equation*}
Here, we used estimates \eqref{eq:step-1-1} and \eqref{eq:step-1-2}.
Consequently, by the definition of limiting essential spectrum (see Definition \ref{def:lim-ess}), we have $\lambda \in \sigess(\seqn{T_{R_n}})$.

\end{proof}

\section{Inclusion for the Essential Spectrum}\label{sec:ess-incl}

Consider the Sturm-Liouville operator $T_0$ introduced in Section \ref{sec:spec-incl}. 
Suppose that the conditions of Assumption \ref{ass:exp} are met. 
As before, fix $\gamma \in \C \bs \set{0}$, define the perturbed operators by 
\begin{equation*}
    T_R u = T_0 u + i \gamma \chi_{[0,R]} u,\, D(T_R) = D(T_0) \qquad (R \in \R_+)
\end{equation*}
and define the limit operator by $\T = T_0 + i \gamma$.

In this section, we prove that the essential spectrum of the limit operator $\T$ is approximated by the eigenvalues of $T_R$ as $R\to \infty$. To achieve this, we require an additional assumption which ensures that the solution $\psi_+$ of $\tilde{T}_0 u =  \lambda u$ introduced in Assumption \ref{ass:exp} can be analytically continued, with respect to the spectral parameter $\lambda$, into an open neighbourhood in $\C$ of any point in the interior of $\sigess(T_0)$. The interior of the essential spectrum is denoted by $\text{int}(\sigess(T_0))$ and defined with respect to the subspace topology.

\begin{assumption}\label{ass:analytic}
$T_0$ is such that $\sigess(T_0) \subseteq \R$. For any $\mu \in \text{int}(\sigess(T_0))$, there exists an open neighbourhood $V_\mu$ of $\mu$ such that:
\begin{enumerate}[label=(\roman*)]
    \item $k$ admits analytic continuations $\kappa_u\,(\kappa_l)$ from the half-planes $\C_+\,(\C_-)$ respectively into $V_\mu$, with 
    \begin{equation}\label{eq:kappa-cond}
        \Im \kappa_u(z),\,-\Im \kappa_l(z) \begin{cases} > 0  & \text{if }z \in \C_+\cap V_\mu \\
        = 0 & \text{if }z \in \R \cap V_\mu \\
        < 0 & \text{if }z \in \C_- \cap V_\mu
        \end{cases}.
    \end{equation}
    \item For any $R > 0$, $\tpsip(R,\cdot)$ admits analytic continuations $\tilde{\varphi}_u(R,\cdot) \, (\tilde{\varphi}_l(R,\cdot))$ from $\C_+ \, (\C_-)$ respectively into $V_\mu$ and $\tpsipd(R,\cdot)$ admits analytic continuations $\tilde{\varphi}_u^d(R,\cdot) \, (\tilde{\varphi}_l^d(R,\cdot)) $ from $\C_+ \,(\C_-)$ respectively into $V_\mu$. $\tilde{\varphi}_j$ and $\tilde{\varphi}_j^d$ satisfy 
    \begin{equation}\label{eq:phi-t-Linf}
        \norm{\tilde{\varphi}_j(\cdot,z)}_{L^\infty(0,\infty)},\, \norm{\tilde{\varphi}_j^d(\cdot,z)}_{L^\infty(0,\infty)} < \infty \qquad (j = u \text{ or }l)
    \end{equation}
    for all $z \in V_\mu$.
    \item For each $z \in V_\mu$, the functions $\varphi_u(\cdot,z)$ and $\varphi_l(\cdot,z)$, defined by
    \begin{equation}\label{eq:phi-j-defn}
    \varphi_j(x,z) := e^{i \kappa_j(z) x} \tilde{\varphi}_j(x,z), \qquad (j = u  \text{ or }l),
    \end{equation}
    solve the equation $\tilde{T}_0 \varphi = z \varphi $ and satisfy 
    \begin{equation}\label{eq:phi-j-deriv}
     \varphi'_j(x,z) = e^{i \kappa_j(z)x} \tilde{\varphi}^d_j(x,z). \qquad (j = u \text{ or }l)
    \end{equation}
\end{enumerate}
\end{assumption}

In the following two examples, by analytic continuations we mean analytic continuations from $\C_+$ and $\C_-$ into $V_\mu$. 

\begin{example}[Schr\"odinger operators with $L^1$ potentials, continued]\label{ex:Lev-3}
    Consider again the case $p = r = 1$ with $q \in L^1(0,\infty)$, introduced in Example \ref{ex:Lev-1}.
    Recall that $k(\lambda) = \sqrt{\lambda}$ so Assumption \ref{ass:analytic} (i) is satisfied in this case. Recall that 
    \begin{equation*}
        \tilde{\psi}_\pm(x,z)=1 + E_\pm(x,z)\text{ and }\tilde{\psi}^d_\pm(x,z)= \pm i \sqrt{z}\br{1 + E^d_\pm(x,z)}.
    \end{equation*}
     In order to show that Assumption \ref{ass:analytic} (ii) and (iii) hold in this case it suffices to show that for any $\mu \in \text{int}(\sigess(T_0))$ and any $x \in [0,\infty)$, $E_+(x,\cdot)$ and $E^d_+(x,\cdot)$ admit analytic continuations $E(x,\cdot)$ and $E^d(x,\cdot)$ (respectively) into an open neighbourhood $V_\mu$ of $\mu$ independent of $x$, such that the function $\varphi(\cdot,z)$ defined by 
     \begin{equation}\label{eq:Lev-cont}
        \varphi(x,z) := e^{i \sqrt{z} x} \br*{ 1 + E(x,z)}
     \end{equation}
     satisfies 
     \begin{equation}\label{eq:Lev-cont-deriv}
        \varphi'(x,z)  = i \sqrt{z} e^{i \sqrt{z} x} \br*{ 1 + E^d(x,z)},
     \end{equation}
     solves the Schr\"odinger equation $ - \varphi'' + q \varphi = z \varphi$ and satisfies 
    \begin{equation*}
        |E_\pm(x,z)|,|E^d_\pm(x,z)|\to 0\text{ as }x \to \infty
    \end{equation*}
    for any fixed $z \in V_\mu$. 
    Note that $\sqrt{\cdot}$ is understood to have been analytically continued into $V_\mu$ in \eqref{eq:Lev-cont} and \eqref{eq:Lev-cont-deriv}.
    Additional conditions on the potential $q$ are required to ensure that this holds.
    Two such conditions are:
    \begin{enumerate}[label=\rm{(\alph*)},wide, labelindent=0pt]
    \item 
    (Naimark condition \cite[Lemma 1]{stepin2015complex}) There exists $a > 0$ such that 
    \begin{equation}
        \int_0^\infty e^{ax}|q(x)| \d x < \infty.
    \end{equation}
    
    \item (Dilation analyticity \cite{brown2002analytic}) $q$ is real-valued and can be analytically continued into some open, convex region $U\subset \C$ containing a sector $\set{z \in \C : \arg(z) \in [-\theta,\theta]}$ for some $\theta \in (0,\frac{\pi}{2}]$. Furthermore, there exists $C_0 > 0$ and $\beta > 1$ independent of $z$ such that 
    \begin{equation}
        |q(z)| \leq C_0 |z|^{-\beta}
    \end{equation}
    for all $z \in U$.
    \end{enumerate}
    
\end{example}

\begin{example}[Eventually periodic Schr\"odinger operators, continued]\label{ex:Floq-3}
Consider again the case $p = r = 1$ with $q$ eventually real periodic, introduced in Example \ref{ex:Floq-1}. 
As mentioned in Example \ref{ex:Floq-1}, for any $\mu \in \text{int}(\sigess(T_0)) = \sigess(T_0) \bs B_{\text{ends}}$ and any $x \in [0,\infty)$, the functions $k$, $\psi_+(x,\cdot)$ and $\psi'_+(x,\cdot)$ admit analytic continuations into an open neighbourhood $V_\mu$ of $\mu$.
\begin{enumerate}[label=\rm{(\roman*)},wide, labelindent=0pt]
\item 
By the expression \eqref{eq:Floq-mult} for $\rho_+$, the analytic continuations $\tilde{\rho}_+$ for $\rho_+$, from $\C_\pm$ into $V_\mu$, satisfies
\begin{equation*}
    |\tilde{\rho}_+| \begin{cases} < 1  & \text{if }z \in \C_\pm\cap V_\mu \\
        = 1 & \text{if }z \in \R \cap V_\mu \\
        > 1 & \text{if }z \in \C_\mp \cap V_\mu
        \end{cases}.
\end{equation*}
Hence, the analytic continuations of $k$ satisfy equation \eqref{eq:kappa-cond}.
\item
 The analytic continuations of $\tilde{\psi}_+(x,\cdot)$ and $\tilde{\psi}^d_+(x,\cdot)$ satisfy the $L^\infty$ estimates \eqref{eq:phi-t-Linf} by their definitions \eqref{eq:Floq-psit} and \eqref{eq:Floq-psit-d}.
\item
The analytic continuations with respect to $z$ of $\psi_+(\cdot,z)$ solve the Schr\"odinger equation $-\psi'' + q \psi = z\psi$ since by \eqref{eq:Floq-soln-eigenvect} they are linear combinations of the solutions $\phi_1(\cdot,z)$ and $\phi_2(\cdot,z)$. 
Expressions \eqref{eq:phi-j-defn} and \eqref{eq:phi-j-deriv} for the analytic continuations of $\psi_+$ and $\psi'_+$ hold by the definition of (the analytic continuations of) $\tilde{\psi}_+$ and $\tilde{\psi}^d_+$ respectively. 

\end{enumerate}
\end{example}

Throughout the remainder of the section, let $\mu \in \text{int}(\sigess(T_0))$ and suppose that the conditions of Assumption \ref{ass:analytic} are satisfied. 
Also, assume without loss of generality that $(i \gamma + V_\mu )\cap \R = \emptyset$. 

\begin{lemma}\label{lem:gR}
$\lambda \in i \gamma + V_\mu$ is an eigenvalue of $T_R$ if and only if 
\begin{equation*}
    g_R(\lambda) := \beta_u(R,\lambda)e^{i \kappa_u(\lambda - i \gamma) R} + \beta_l(R,\lambda)e^{i \kappa_l(\lambda - i \gamma) R} = 0
\end{equation*}
where 
\begin{equation*}
    \beta_u(R,\lambda)  := BC[\varphi_l(\cdot,\lambda - i \gamma)] \br*{ \tilde{\varphi}_u(R,\lambda - i \gamma) \tpsipd(R, \lambda) - \tilde{\varphi}^d_u(R, \lambda - i \gamma) \tpsip(R, \lambda) } 
\end{equation*}
and 
\begin{equation*}
    \beta_l(R,\lambda)  := BC[\varphi_u(\cdot,\lambda - i \gamma)] \br*{ \tpsip(R,\lambda) \tilde{\varphi}_l^d(R, \lambda - i \gamma) - \tpsipd(R,\lambda) \tilde{\varphi}_l(R, \lambda - i \gamma) }.
\end{equation*}
Furthermore, $g_R$ is analytic on $i \gamma + V_\mu$. 
\end{lemma}

\begin{proof}
The proof is similar to the proof of Lemma \ref{lem:fR}.

Let $\lambda \in i \gamma + V_\mu $. Any solution of the boundary value problem 
\begin{equation*}
    (\tilde{T}_0 + i \gamma \chi_{[0,R]})u = \lambda u \text{ on }[0,R],\,BC[u] = 0
\end{equation*}
must lie in $\text{span}_\C\set{u_1(\cdot,\lambda)}$, where $u_1$ is defined by 
\begin{equation*}
    u_1(x,\lambda) = BC[\varphi_l(\cdot,\lambda - i \gamma)] \varphi_u(x, \lambda - i \gamma) - BC[\varphi_u(\cdot,\lambda - i \gamma)] \varphi_l(x, \lambda - i \gamma).
\end{equation*}
Since $(i \gamma + V_\mu )\cap \R = \emptyset$, any $L^2_r$ solution of $(\tilde{T}_0 +i \gamma \chi_{[0,R]})u = \lambda u$ on $[R,\infty)$ must lie in $\text{span}_\C \set{\psip(\cdot,\lambda)}$. 
$\lambda$ is an eigenvalue if and only if 
\begin{equation*}
    u_1(R,\lambda) \psip'(R,\lambda) - u_1'(R, \lambda) \psip(R, \lambda) = 0 
\end{equation*}
which holds if and only if $g_R(\lambda) = 0$.
\end{proof}

We proceed on to the proof of inclusion for the essential spectrum of $\T$, which consists in proving that there exists eigenvalues of $T_R$ accumulating to $\mu + i \gamma$ as $R \to \infty$. 
We can only achieve this with the additional assumption that $\mu + i \gamma$ does not lie in a region of the complex plane in which either $\beta_u(R,\cdot)$ or $\beta_l(R,\cdot)$ become small as $R \to \infty$.
We now define a subset of the complex plane capturing such regions. 
\begin{definition}
Define a subset $\Sr \subset \C$ by
\begin{equation}\label{eq:Sr-defn}
    \Sr = \set*{ \lambda \in i \gamma + V_\mu \cap \R : \liminf_{R \to \infty} |\beta_j(R,\lambda)| = 0,\,j = u \text{ or }l }.
\end{equation}  
\end{definition}

The strategy of the proof is to first introduce an approximation $g_R^\infty(\lambda)$ to $g_R(\lambda)$ which is valid for $\lambda$ near $\mu + i \gamma$. 
It is then shown that there exists zeros $\lambda_R^\infty$ of $g_R^\infty$  converging to $\mu + i \gamma$ as $R \to \infty$. 
A family of simple closed contours $\ell_R$ surrounding $\lambda_R^\infty$ are constructed such that $\text{dist}(\ell_R, \mu + i \gamma) \to 0 $ as $R \to \infty$.
We estimate $|g_R^\infty|$ from below and $|g_R - g_R^\infty|$ from above on $\ell_R$ to conclude, using Rouch\'e's Theorem, that there exists a zero $\lambda_R$ of $g_R$ inside $\ell_R$ for all large enough $R$. 
Such $(\lambda_R)$ would be eigenvalues of $T_R$ and would converge to $\mu + i \gamma$ as $R \to \infty$, giving the result.

\begin{lemma}\label{lem:kappa-inverse}
 
The function $\kappa_u - \kappa_l$ has an analytic inverse $(\kappa_u - \kappa_l)^{-1}:B_{\delta}(w_0) \to \C$ for some small enough $\delta > 0$, where $w_0 := (\kappa_u - \kappa_l)(\mu)$,
 
\end{lemma}

\begin{proof}
 Let $h = \kappa_u - \kappa_l - w_0$. Let $\epsilon > 0$ be small enough so that $|h| > 0$ on $\partial B_\epsilon(\mu)$. 
Assumption \ref{ass:analytic} (i) implies that any $z \in \partial B_\epsilon(\mu)$ satisfies
\begin{equation}\label{eqpr:analytic-1}
    \arg \br*{ \frac{h}{|h|}(z)} = \arg \br*{h(z)} \in \begin{cases}  (0,\pi) &  \text{if }z \in \C_+ \cap V_\mu \\
     \set{0,\pi} & \text{if }z \in \R \cap V_\mu \\
     (\pi,2\pi) &  \text{if }z \in \C_- \cap V_\mu \end{cases} .
\end{equation}
Note that $\text{arg}$ is set so that $\arg(z) = 0$ if $z \in \R_+$.
The topological degree (i.e. the winding number) of the map $h/|h|: \partial B_\epsilon(0) \to \partial B_1(0)$ is equal to the number of zeros for $h$ in $B_\epsilon(0)$, counted with multiplicity \cite[pg. 110]{guillemin2010differential}. 
\eqref{eqpr:analytic-1} implies that the topological degree of $h/|h|$ can only be 1, hence $\mu$ is a simple zero of $\kappa_u - \kappa_l$. 
The lemma now follows from the inverse function theorem. 

\end{proof}

\begin{theorem}\label{thm:ess-incl}
Assume that $ \mu \in \text{\rm{int}} (\sigess(T_0))$ is such that $\mu + i \gamma \notin \Sr$. There exists eigenvalues $\lambda_R$ of $T_R$ $(R \in \R_+)$ and a constant $C_0 = C_0(T_0,\gamma,\mu) > 0$ such that 
\begin{equation*}
    |\lambda_R - (\mu + i \gamma)| \leq \frac{C_0}{R}    
\end{equation*}
for all large enough $R$.
\end{theorem}

\begin{proof}
Let $C>0$ be an arbitrary constant independent of $R$ and $\theta$.

Define approximation $g_R^\infty$ to $g_R$ by 
\begin{equation*}
    g_R^\infty(\lambda) = \beta_{u,R} e^{i \kappa_u(\lambda - i \gamma) R} - \beta_{l,R} e^{i \kappa_l(\lambda - i \gamma) R}
\end{equation*}
where $\beta_{u,R} := \beta_u(R,\mu + i \gamma)$ and $\beta_{l,R} := - \beta_l(R,\mu + i \gamma)$. By the definition of $\Sr$, the $L^\infty$ estimates \eqref{eq:psi-t-Linf} of Assumption \ref{ass:exp} (ii) and the $L^\infty$ estimates \eqref{eq:phi-t-Linf} of Assumption \ref{ass:analytic} (ii), there exists $C_1, C_2 > 0$ independent of $R$ such that $\beta_{u,R}$ and $\beta_{l,R}$ satisfy 
\begin{equation}\label{eqpr:analytic-0}
    C_1 \leq |\beta_{j,R}| \leq C_2 \qquad(j = u \text{ or }l)
\end{equation}
for all large enough $R$. $g_R^\infty(\lambda) = 0$ holds if and only if 
\begin{equation}
    (\kappa_u - \kappa_l)(\lambda - i \gamma) = - \frac{i}{R}\br*{ \ln\br*{\frac{\beta_{l,R}}{\beta_{u,R}}} + 2 \pi i n } =: \tilde{\kappa}(n)
\end{equation}
for some $n \in \mathbb{Z}$.

Let $w_0 := (\kappa_u - \kappa_l)(\mu)$ and $n(R) := \floor{R w_0/(2\pi)}$. 
Note that $n(R)$ is well-defined since $\mu \in \R$ and  $\Im w_0 = 0$ by Assumption \ref{ass:analytic}.
Using \eqref{eqpr:analytic-0}, 
\begin{equation}\label{eqpr:analytic-kap-tild}
 |\tilde{\kappa}(n(R)) - w_0| \leq \frac{1}{R} \abs*{ \ln\br*{\frac{\beta_{l,R}}{\beta_{u,R}}}} + \abs*{ \frac{2 \pi n(R)}{R} - w_0 } \leq \frac{C}{R}
\end{equation}
for large enough $R$.
By Lemma \ref{lem:kappa-inverse}, there exists an analytic inverse $(\kappa_u - \kappa_l)^{-1}:B_{2\delta}(w_0) \to \C$ for some small enough $\delta > 0$.
Let $R_0 > 0$ be large enough such that $\tilde{\kappa}(n(R))$ lies in $B_{\delta}(w_0)$ for all $R \geq R_0$. Define 
\begin{equation}
    \lambda_R^\infty = (\kappa_u - \kappa_l)^{-1}(\tilde{\kappa}(n(R))) + i \gamma\qquad (R \geq R_0). 
\end{equation}
Then $g_R^\infty(\lambda_R^\infty) = 0$ and, by the analyticity of $(\kappa_u - \kappa_l)^{-1}$ as well as \eqref{eqpr:analytic-kap-tild},
\begin{equation}\label{eqpr:analytic-lamRinf-est}
    |\lambda_R^\infty - (\mu + i \gamma)| \leq C |\tilde{\kappa}(n(R)) - w_0| \leq \frac{C}{R}
\end{equation}
for large enough $R$.
For $R \geq R_0$, define family $\ell_R = \set{\ell_R(\theta): \theta \in [0,2\pi)}$ of simple closed contours around $\lambda_R^\infty $ by 
\begin{equation}
    \ell_R(\theta) = (\kappa_u - \kappa_l)^{-1}(\tilde{\kappa}(n(R)) + \frac{\delta}{R}e^{i \theta}) + i \gamma.
\end{equation}
By the analyticity of $(\kappa_u - \kappa_l)^{-1}$ and estimate \eqref{eqpr:analytic-lamRinf-est}, we have that
\begin{equation}\label{eqpr:analytic-2}
    |\ell_R(\theta) - (\mu + i \gamma)| \leq |\ell_R(\theta) - \lambda_R^\infty| + |\lambda_R^\infty - (\mu + i \gamma)| \leq \frac{C}{R}
\end{equation}
for large enough $R$.

By a direct computation, we have
\begin{equation}\label{eqpr:analytic-3}
    e^{i \kappa_u(\ell_R(\theta) - i \gamma)R} = \frac{\beta_{l,R}}{\beta_{u,R}}e^{i \delta e^{i \theta}} e^{i \kappa_l(\ell_R(\theta) - i \gamma)R}.
\end{equation}
By Assumption \ref{ass:analytic} (ii), $\beta_u(R,\cdot)$ and $\beta_l(R,\cdot)$ are analytic and bounded in $R$ uniformly in a small enough neighbourhood of $\mu + i \gamma$, so, using the Cauchy integral formula as in \eqref{eq:power-method} and using \eqref{eqpr:analytic-2},
\begin{equation}\label{eqpr:analytic-4}
 |\beta_j(R,\ell_R(\theta)) - \beta_j(R,\mu + i \gamma)| \leq C|\ell_R(\theta) - (\mu + i \gamma)| \leq \frac{C}{R} \qquad (j = u \text{ or } l) 
\end{equation}
for large enough $R$.
Using \eqref{eqpr:analytic-0}, \eqref{eqpr:analytic-3} and \eqref{eqpr:analytic-4},
\begin{align*}
    & |g_R(\ell_R(\theta)) - g_R^\infty(\ell_R(\theta))|  \\
    & \qquad \leq \br*{|\beta_u(R,\ell_R(\theta)) - \beta_{u,R}|\abs*{\frac{\beta_{l,R}}{\beta_{u,R}}e^{i \delta e^{i \theta}}} + |\beta_l(R,\ell_R(\theta)) + \beta_{l,R}|} e^{- \Im \kappa_l(\ell_R(\theta) - i \gamma) R} \\
    & \qquad \leq \frac{C}{R} e^{- \Im \kappa_l(\ell_R(\theta) - i \gamma) R}
\end{align*}
for large enough $R$. 
Similarly, 
\begin{align*}
    |g_R^\infty(\ell_R(\theta))| = |\beta_{l,R}|\abs*{e^{i \delta e^{i \theta}} - 1}e^{- \Im \kappa_l(\ell_R(\theta) - i \gamma) R} \geq C e^{- \Im \kappa_l(\ell_R(\theta) - i \gamma) R}.
\end{align*}
For each large enough $R$, Rouch\'e's condition 
\begin{equation*}
    |g_R(\ell_R(\theta)) - g_R^\infty(\ell_R(\theta))| < |g_R^\infty(\ell_R(\theta))|
\end{equation*}
is satisfied so there exists a zero $\lambda_R$ of $g_R$ in the interior of $\ell_R$ such that 
\begin{equation*}
    |\lambda_R - (\mu + i \gamma)| \leq | \lambda_R - \lambda_R^\infty| + |\lambda_R^\infty - (\mu + i \gamma)| \leq \frac{C_0}{R}
\end{equation*}
for some $C_0>0$ independent of $R$.
\end{proof}

We finish this section with a characterisation of the set $\Sr$ in the case that $T_0$ is a  Schr\"odinger operator with an $L^1$ or an eventually real periodic potential.

\begin{definition}

Define function $\Lambda_u:[0,\infty)\times (i \gamma + V_\mu) \to \C $ by
\begin{equation}\label{eq:Lam-u-defn}
\Lambda_u(R,\lambda) = \tilde{\varphi}_u(R,\lambda - i \gamma) \tpsipd(R, \lambda) - \tilde{\varphi}^d_u(R, \lambda - i \gamma) \tpsip(R, \lambda)
\end{equation}
and define function $\Lambda_l:[0,\infty)\times (i \gamma + V_\mu) \to \C $ by
\begin{equation}\label{eq:Lam-l-defn}
\Lambda_l(R,\lambda) =  \tpsip(R,\lambda) \tilde{\varphi}_l^d(R, \lambda - i \gamma) - \tpsipd(R,\lambda) \tilde{\varphi}_l(R, \lambda - i \gamma). 
\end{equation}
\end{definition}
By the definition of $\beta_u$ and $\beta_l$ in Theorem \ref{lem:gR},
\begin{equation*}
    \beta_u(R,\lambda) = BC[\varphi_l(\cdot,\lambda - i \gamma)]\Lambda_u(R,\lambda)
    \text{ and }
    \beta_l(R,\lambda) = BC[\varphi_u(\cdot,\lambda - i \gamma)]\Lambda_l(R,\lambda)
\end{equation*}
hence we have the following characterisation of $\Sr$:

\begin{corollary}
$\Sr$ can be decomposed as
\begin{equation}
    \Sr = (i \gamma + S_{\mathfrak{r},0}) \cup S_{\mathfrak{r},u} \cup S_{\mathfrak{r},l}
\end{equation}
where
\begin{equation}\label{eq:Sr-0-defn}
    S_{\mathfrak{r},0} := \set*{z \in V_\mu \cap \R: BC[\varphi_u(\cdot,z)] = 0\text{ or } BC[\varphi_l(\cdot,z)] = 0}
\end{equation}
and 
\begin{equation}\label{eq:Sr-j-defn}
    S_{\mathfrak{r},j} := \set*{\lambda \in i \gamma + V_\mu \cap \R: \liminf_{R \to \infty} |\Lambda_j(R,\lambda)| = 0}\qquad(j = u\text{ or }l).
\end{equation} 
\end{corollary}

The elements of $S_{\mathfrak{r},0}$ are precisely the resonances of $T_0$ in $V_\mu \cap \R$, by definition.

\begin{example}[Schr\"odinger operators with $L^1$ potentials, continued]\label{ex:Lev-4}
Consider again the case $p = r = 1$ with $q \in L^1(0,\infty)$ satisfying the necessary conditions ensuring that Assumption \ref{ass:analytic} holds, as discussed in Example \ref{ex:Lev-3}.
In this case, since the functions $E_\pm(R,\lambda)$ and $E^d_\pm(R,\lambda)$ tend to zero as $R\to \infty$ for any $\lambda$, $\Lambda_u$ and $\Lambda_l$ satisfy 
\begin{equation*}
    \abs{\Lambda_j(R,\lambda)} \to \abs*{ \sqrt{\lambda - i \gamma} - \sqrt{\lambda}}\text{ as }R \to \infty \qquad(j = u\text{ or }l)    
\end{equation*}
for all $\lambda \in i \gamma + V_\mu$, where the square-root is understood to have been analytically continued from $\C_+$ ($\C_-$) into $V_\mu$ in the case $j = u$ ($j = l$) respectively. Since $\sqrt{\lambda - i \gamma} \neq  \sqrt{\lambda}$ for all $\lambda \in i \gamma + V_\mu$, regardless of which branch-cut for the square-root is chosen, we have
\begin{equation*}
    S_{\mathfrak{r},u} = S_{\mathfrak{r},l} = \emptyset. 
\end{equation*}
Consequently, 
\begin{equation*}
    \Sr = i\gamma + S_{\mathfrak{r},0},
\end{equation*}
that is, $\mu + i \gamma \in \Sr$ if and only if $\mu$ is a resonance of $T_0$
\end{example}

\begin{example}[Eventually periodic Schr\"odinger operators, continued]\label{ex:Floq-4}

Consider the case $p = r = 1$ with $q$ real-valued and $q|_{[X,\infty)}$ $a$-periodic for some $X \geq 0$ and $a > 0$. Assume that $\eta \in  [0,\pi)$, so that $T_0$ is equipped with a real mixed boundary condition at 0. Note that $T_0$ is self-adjoint in this case. $q$ is eventually real periodic so by Example \ref{ex:Floq-3}, Assumption \ref{ass:analytic} is satisfied. 
The sets $S_{\mathfrak{r},u}$ and $S_{\mathfrak{r},l}$ satisfy 
\begin{equation}\label{eq:empty-Floq-4}
    S_{\mathfrak{r},u} = S_{\mathfrak{r},l} = \emptyset. 
\end{equation}
Consequently, 
\begin{equation*}
    \Sr = i\gamma + S_{\mathfrak{r},0},
\end{equation*}
that is, $\mu + i \gamma \in \Sr$ if and only if $\mu$ is a resonance of $T_0$

\end{example}

\begin{proof}[Proof of \eqref{eq:empty-Floq-4}]

We will only prove \eqref{eq:empty-Floq-4} for $j=u$, the proof for $j = l$ is similar. 

Assume for contradiction that $S_{\mathfrak{r},u}$ is non-empty and let $\lambda \in S_{\mathfrak{r},u}$. By unique continuation, expressions analogous to \eqref{eq:Floq-psit} and \eqref{eq:Floq-psit-d} hold for $\tilde{\varphi}_u$ and $\tilde{\varphi}^d_u$. 
By these expressions, there exists a sequence $\seqn{x_{0,n}} \subset [X,X + a)$ such that $\Lambda_u(x_{0,n},\lambda) \to 0$ as $n \to \infty$.
Let $x_0$ be any accumulation point of $\seqn{x_{0,n}}$. 
Then, since $\Lambda_u(\cdot,\lambda)$ is absolutely continuous, it holds that $\Lambda_u(x_0,\lambda) = 0$, so, 
\begin{equation}\label{eqpr:Floq-4-1}
    \varphi_u(x_0,\lambda - i \gamma) \psi'_+(x_0,\lambda) - \varphi'_u(x_0,\lambda - i \gamma) \psi_+(x_0, \lambda) = 0.
\end{equation}
Noting that the solutions $\phi_1(\cdot,\lambda - i \gamma)$ and $\phi_2(\cdot,\lambda - i \gamma)$ defined by \eqref{eq:phi-1-2} are real since $\lambda - i \gamma \in \R$ and that the analytic continuations $\rho_u(\rho_l)$ for $\rho_+$ from $\C_+(\C_-)$ respectively satisfy $\overline{\rho_u(\lambda - i \gamma)} = \rho_l(\lambda - i \gamma)$, the expression analogous to \eqref{eq:Floq-soln-eigenvect} for the Floquet solution $\varphi_u$ implies that
\begin{equation*}
    \overline{\varphi_u}(x,z) = - \phi_2(X+a,z) \phi_1(x,z) +  (\phi_1(X+a,z) - \overline{\rho_u}(z)) \phi_2(x,z) = \varphi_l(x,z) 
\end{equation*}
where $z := \lambda - i \gamma$.
Consequently we have,
\begin{equation}\label{eqpr:Floq-4-2}
    \varphi_l(x_0,\lambda - i \gamma) \overline{\psi'_+}(x_0,\lambda) - \varphi'_l(x_0,\lambda - i \gamma) \overline{\psi_+}(x_0, \lambda) = 0.
\end{equation}
By \eqref{eqpr:Floq-4-1} and \eqref{eqpr:Floq-4-2}, there exists $C_{1,u},C_{2,u},C_{1,l},C_{2,l} \in \C \bs \set{0}$ independent of $x$ such that the functions
\begin{equation*}
    u_u(x,\lambda ) := \begin{cases} C_{1,u} \varphi_u(x,\lambda - i \gamma) & \text{if }x \in [0,x_0) \\
    C_{2,u} \psi_+(x,\lambda) & \text{if }x \in [x_0,\infty)
    \end{cases}
\end{equation*}
and 
\begin{equation*}
    u_l(x,\lambda ) := \begin{cases} C_{1,l} \varphi_l(x,\lambda - i \gamma) & \text{if }x \in [0,x_0) \\
    C_{2,l} \overline{\psi_+}(x,\lambda) & \text{if }x \in [x_0,\infty)
    \end{cases}
\end{equation*}
are absolutely continuous and solve the Schr\"odinger equation $\tilde{T}_{x_0} u = \lambda u$.
Note that $\overline{\psi_+}$ solves the Schr\"odinger equation $\tilde{T}_{x_0} u = \lambda u$ on $[x_0,\infty)$ because $q$ is real-valued. 
By orthogonality, there exists $(a_u,a_l) \in \C^2 \bs \set{(0,0)}$ such that 
\begin{equation*}
    BC[a_u u_u(\cdot,\lambda) + a_l u_l(\cdot,\lambda)] = a_u C_{1,u} BC[\varphi_u(\cdot,\lambda - i \gamma)] + a_l C_{1,l} BC[\varphi_l(\cdot,\lambda - i \gamma)] = 0
\end{equation*}
This implies that $\lambda$ is an eigenvalue of $T_{x_0}$ with corresponding eigenfunction $u:=a_u u_u + a_l u_l$. 
By a standard integration by parts, 
\begin{equation*}
    \Im(\lambda) = \gamma\frac{\int_0^{x_0}|u|^2}{\int_0^\infty|u|^2} < \gamma
\end{equation*}
which is the desired contradiction.
\end{proof}

\section{Numerical examples}\label{sec:num-ex}

In this section, we illustrate the results from Sections \ref{sec:spec-incl} and \ref{sec:ess-incl} with numerical examples. 

\begin{figure}[t]
\centering
\includegraphics[width=\linewidth]{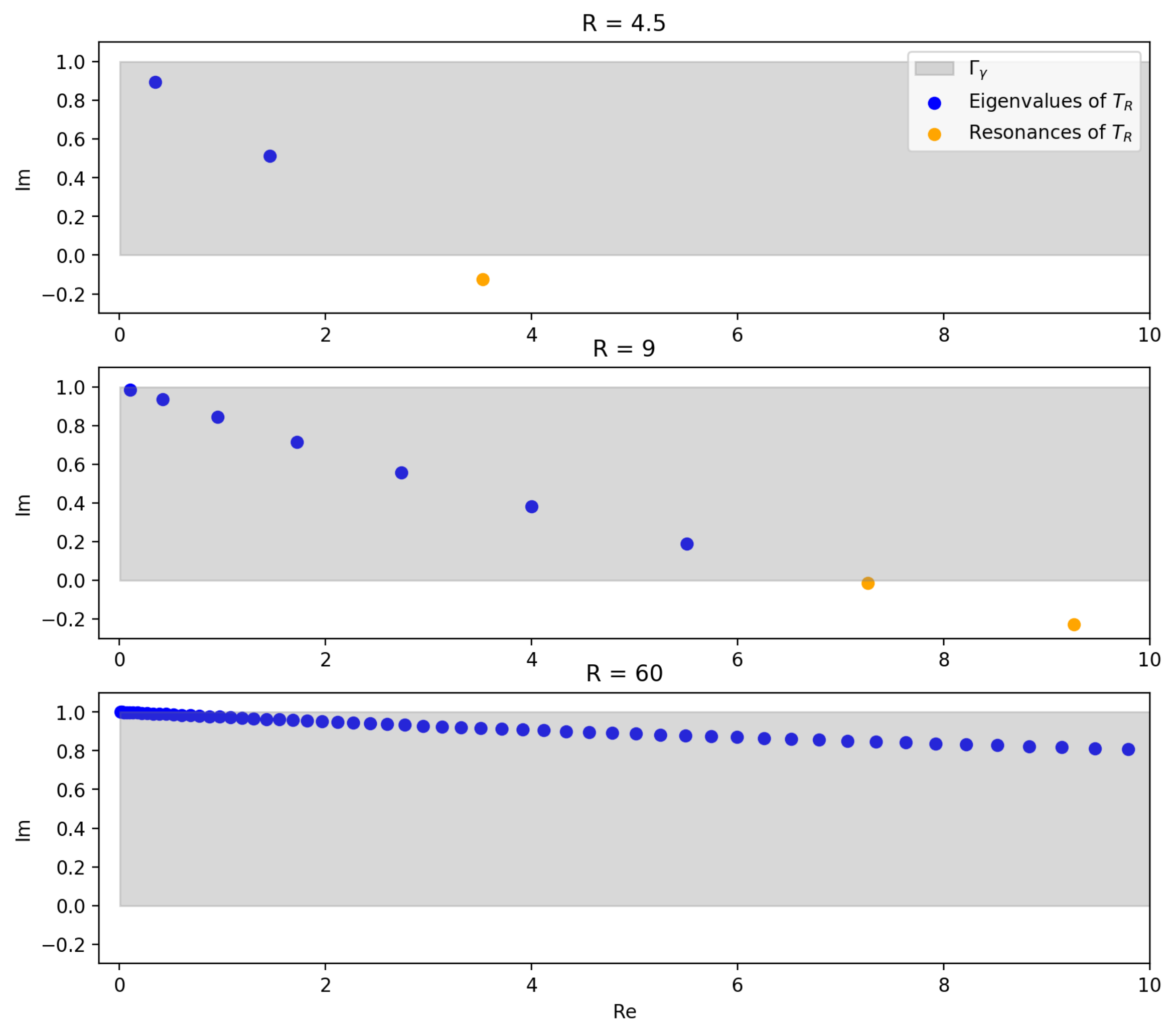}
\caption{Plot of the eigenvalues and resonances of the operator $T_R$  defined by \eqref{eq:ex-1-TR}.}
\label{fig:1}
\end{figure}

\begin{example}\label{ex:num-1}

Consider perturbed operators of the form 
\begin{equation}\label{eq:ex-1-TR}
    T_R = - \frac{\d^2}{\d x^2} + i \chi_{[0,R]}(x) \qquad (R \in \R_+)
\end{equation}
endowed with Dirichlet boundary conditions at 0. 
This corresponds to the case $p = r = 1 $, $q = 0$, $\eta = 0$ and $\gamma = 1$ in Sections \ref{sec:spec-incl} and \ref{sec:ess-incl}.

By an explicit computation, $\lambda \in \C \bs [0,\infty)$ is an eigenvalue of $T_R$ if and only if
\begin{equation}\label{eq:ex-1-fR}
f_R(\lambda) = i \sqrt{\lambda} \sin ( \sqrt{\lambda - i} R ) - \sqrt{\lambda - i } \cos ( \sqrt{\lambda - i } R ) = 0.
\end{equation}
Note that our convention is that the branch cut of the square-root is along $[0,\infty)$. 
By suitably analytically continuing the square root function in \eqref{eq:ex-1-fR}, any $\lambda$ in the lower right quadrant of the complex plane is a resonance of $T_R$ if and only if $f_R(\lambda) = 0$. 

To numerically compute the zeros of $f_R$, hence the eigenvalues and resonances of $T_R$, in a fixed bounded region, we use a Python implementation of an algorithm utilising the argument principle \cite{dellnitz2002locating}. The results are illustrated in Figure \ref{fig:1}.

For small enough $R > 0$, $T_R$ has no eigenvalues \cite{frank2016number}. 
As $R$ increased, we observe resonances in the lower half plane emerging out of $\sigess(T_R) = [0,\infty)$, to become eigenvalues in the numerical range
\begin{equation*}
\Gamma_\gamma := \sigess(T_0) \times i [0,\gamma] = [0,\infty)\times i[0,\gamma]      
\end{equation*}
of $T_0$ accumulating to $\sigess(T) = i \gamma + [0,\infty)$, as expected by Theorem \ref{thm:ess-incl}.

\end{example}

\begin{figure}[t]
\centering
\includegraphics[width=\linewidth]{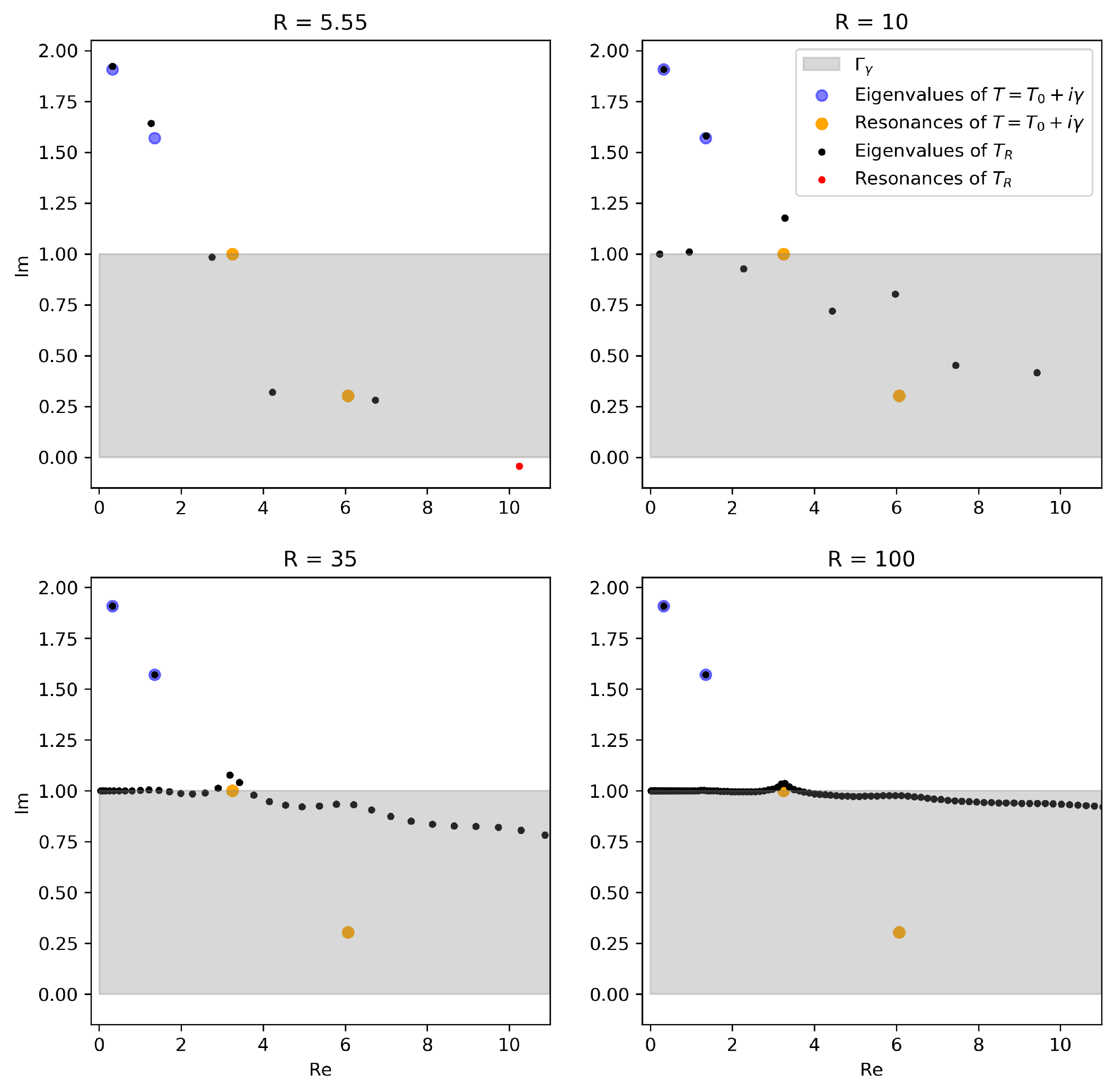}
\caption{Plot of eigenvalues and resonances of the operators $T_R$ and $T= T_0 + i\gamma$ defined by \eqref{eq:ex-2-TR}, with $R_0 = 4.7$.}
\label{fig:2}
\end{figure}

\begin{example}\label{ex:num-2}

Consider perturbed operator of the form
\begin{equation}\label{eq:ex-2-TR}
    T_R = T_0 + i \chi_{[0,R]}(x) = - \frac{\d^2}{\d x^2} + i \chi_{[0,R_0]}(x) + i \chi_{[0,R]}(x)\qquad(R \in \R_+)
\end{equation}
endowed with Dirichlet boundary conditions at 0.
This corresponds to the case $p = r = 1$, $\eta = 0$, $q = i \chi_{[0,R_0]}$ for some $R_0 > 0$ and $\gamma =1$ in Sections \ref{sec:spec-incl} and \ref{sec:ess-incl}. 

By an explicit computation, $\lambda \in \C \bs [0,\infty)$ is an eigenvalue of $T_R$ if and only if
\begin{multline}\label{eq:ex-2-fR}
f_R(\lambda) =  i\sqrt{\lambda - i }\left[ e^{-2i\sqrt{\lambda - i }(R - R_0)} - \frac{\sqrt{\lambda - i} - \sqrt{\lambda }}{\sqrt{\lambda - i} + \sqrt{\lambda}} \right] \sin ( \sqrt{\lambda - 2 i  } R_0 ) \\
- \sqrt{\lambda - 2 i } \left[ e^{-2 i \sqrt{\lambda - i }(R - R_0)} + \frac{\sqrt{\lambda - i} - \sqrt{\lambda}}{\sqrt{\lambda - i} + \sqrt{\lambda}} \right] \cos(\sqrt{\lambda - 2 i }R_0) = 0
\end{multline}
As before, by suitably analytically continuing the square root function in \eqref{eq:ex-2-fR}, any $\lambda$ in the lower right quadrant of the complex plane is a resonance of $T_R$ if and only if $f_R(\lambda) = 0$. 

A numerical computation of the zeros of $f_R$, hence the eigenvalues and resonances of $T_R$ is shown in Figure \ref{fig:2}.
We observe that there are eigenvalues of $T_R$ converging rapidly to the eigenvalues of $T$ and that eigenvalues of $T_R$ accumulate to $\sigess(T) = i \gamma + [0,\infty)$, as expected by Theorems \ref{thm:exp-conv} and \ref{thm:ess-incl}.

Recall that Example \ref{ex:Floq-4} guarantees that the rate of convergence of eigenvalues of $T_R$ to $\mu \in \text{int}(\sigess(T)) = i \gamma + (0,\infty)$ is $O(1/R)$, unless $\mu$ is a resonance of $T$.
The limit operator $T$ for our choice of parameters has a resonance embedded in $\sigess(T)$.
We seem to observe a distinction between the way the eigenvalues of $T_R$ accumulate to the resonance compared to other points in the interior of $\sigess(T)$. 
It seems reasonable to conjecture that the rate of convergence to embedded resonances is indeed slower that $O(1/R)$.

\end{example}

\begin{figure}[t]
\centering
\includegraphics[width=\linewidth]{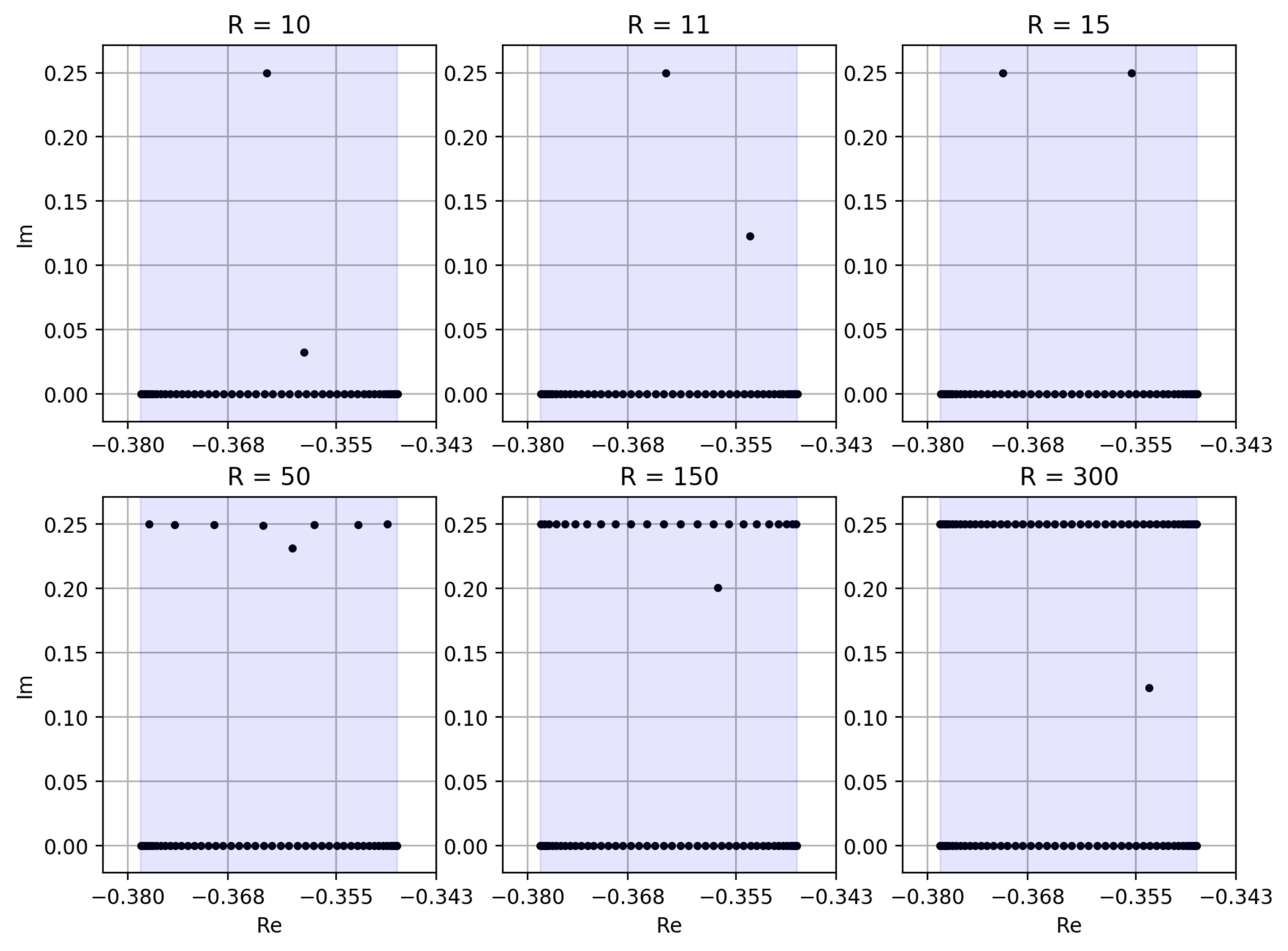}
\caption{ Plot of eigenvalues of the domain truncation and finite difference approximation $T_{R,X,h}$ of the operator $T_R$ defined by \eqref{eq:ex-3-TR}. $h = 0.05$ and $X - R = 300$ are fixed. The region $B \times i \R$ is shaded in light blue.}\label{fig:3}
\end{figure}

\begin{example}\label{ex:num-3}

Consider perturbed operators of the form 
\begin{equation}\label{eq:ex-3-TR}
    T_R = T_0 + \frac{i}{4}\chi_{[0,R]}(x) = - \frac{\d^2}{\d x^2} + \sin(x) + \frac{i}{4}\chi_{[0,R]}(x) \qquad (R \in \R_+)
\end{equation}
endowed with a Dirichlet boundary condition at 0.
This corresponds to the case $p = r = 1$, $\eta = 0$, $q(x) = \sin(x)$ and $\gamma = \frac{1}{4}$ in Section \ref{sec:spec-incl} and \ref{sec:ess-incl}. 
The essential spectrum of $T_0$ has a band gap structure - the first spectral band, which we denote by $B$, is approximately $[-0.3785,-0.3477]$ \cite[Example 15]{marletta2012eigenvalues}.

To numerically compute the eigenvalues of $T_R$, we first perform a domain truncation onto an interval $[0,X]$, imposing a Dirichlet boundary condition at $X$. 
Applying a finite difference method with step-size $h$, we obtain a finite matrix $T_{R,X,h}$.
For fixed $R$, the eigenvalues of $T_{R,X,h}$ accumulate to every point in $\sigma(T_R)$ as $X \to \infty$ and $h\to 0$.
Moreover, any point of accumulation that does not lie in $\sigma(T_R)$ must lie on the real-line (see \cite{chatelin1983spectral} and \cite{marletta2012eigenvalues}).

For a fixed small value of $h$, a fixed large value of $X - R$, the eigenvalues of $T_{R,X,h}$ for increasing $R$ are plotted in Figure \ref{fig:3}. 
We first observe an accumulation of eigenvalues of $T_{R,X,h}$ to the interval $B$ in $\R$. 
These eigenvalues of $T_{R,X,h}$ are due to the domain truncation method approximating $\sigess(T_R)$ and should not be interpreted as approximations of the eigenvalues of $T_R$. 
All other points in the plots are approximations of the eigenvalues of $T_R$.

In Figure \ref{fig:3}, we observe that as $R$ increases, eigenvalues of $T_R$ emerge out of the spectral band $B$ and tend to the shifted spectral band $i \gamma + B$, which is a subset of $\sigess(T)$.
For large $R$, we observe an accumulation of eigenvalues to $i \gamma + B$. 
The eigenvalues of $T_R$ accumulating to $i \gamma + B$ seem to be contained in $B \times i(0,\gamma)$. 
If this is indeed the case then by Bolzano-Weiestrass we expect that there is spectral pollution in $B \times i(0,\gamma)$

\end{example}

\subsection*{Acknowledgements} 
The author would like to express his gratitude to his PhD supervisors Jonathan Ben-Artzi and Marco Marletta, for helpful discussion and guidance. 
The author's research is supported by the United Kingdom Engineering and Physical
Sciences Research Council, through its Doctoral Training Partnership with Cardiff University.

\subsection*{Data Availability}

Data available on request from the author.

\bibliography{DBPollrefs.bib}
\bibliographystyle{plain}

\end{document}